\documentclass[12pt]{article}   	% use "amsart" instead of "article" for AMSLaTeX format
\usepackage{geometry}                		% See geometry.pdf to learn the layout options. There are lots.
\usepackage{graphicx}				% Use pdf, png, jpg, or eps§ with pdflatex; use eps in DVI mode
\usepackage{caption}
\usepackage{subcaption}
\usepackage{amssymb}
\usepackage{amsthm}
\usepackage{csquotes}
\usepackage{amsmath}
\DeclareMathOperator\arctanh{arctanh}
\usepackage{mathtools}
\usepackage{float}
\usepackage{physics}
\usepackage{listings}
\MakeOuterQuote{"}
\newtheorem*{theorem*}{Theorem}
\newtheorem{definition}{Definition}
\newtheorem{theorem}{Theorem}[section]
\newtheorem{proposition}[theorem]{Proposition}
\newtheorem{lemma}[theorem]{Lemma}
\newtheorem*{lemma*}{Lemma}

\newtheorem{remark}[theorem]{Remark}
\newtheorem{corollary}[theorem]{Corollary}
\newtheorem{conjecture}[theorem]{Conjecture}
\newtheorem*{conjecture*}{Conjecture}
\title{An Interpolation from Sol to Hyperbolic Space}
\author{Matei P. Coiculescu}
\begin{document}
\maketitle
\begin{abstract}
We study a one-parameter family of nonisomorphic solvable Lie groups, which, when equipped with canonical left-invariant metrics,
$$ds^2=e^{-2z}dx^2+e^{2\alpha z}dy^2+dz^2$$
becomes an interpolation from a model of the Sol geometry to a model of Hyperbolic Space, with a stop at $\mathbb{H}^2\cross\mathbb{R}$. These Lie groups are also Bianchi groups of Type VI with orthogonal coordinates. As a continuation of joint work with Richard Schwartz on Sol, we primarily analyze those Lie groups in our interpolation with some positive sectional curvature. Our main result is a characterization of the cut locus at the identity of the group that maximizes scalar curvature.
\end{abstract}
\section{Introduction}
We study a one-parameter family of homogeneous Riemannian 3-manifolds that interpolates between three Thurston geometries: Sol, $\mathbb{H}^2\cross\mathbb{R}$, and $\mathbb{H}^3$. Sol is quite strange from a geometric point of view. For example, it is neither rotationally symmetric nor isotropic, and, since Sol has sectional curvature of both signs, there is an interplay between focus and dispersion that causes the Riemannian exponential map to be singular. On the other hand, since both $\mathbb{H}^2\cross\mathbb{R}$ and $\mathbb{H}^3$ have nonpositive sectional curvature, the exponential map at any point is a diffeomorphism onto the whole space.  In this article, we attempt to show that Sol's peculiarity can be slowly untangled by an interpolation of geometries until we reach $\mathbb{H}^2\cross\mathbb{R}$, which has a qualitatively "better" behavior than Sol. The interpolation continues on to $\mathbb{H}^3$, but we will not spend much time with that part of the family. 

The study of families of geometric structures is growing in application. Although done with a different aim, Steve Trettel's Ph.D. thesis \cite{ST} takes this approach. It describes a collection of geometric transitions, defined by constructing analogs of familiar geometries (projective geometry, hyperbolic geometry, etc.) over real algebras. The similarity with our work is the concept of investigating well-known geometric structures as part of a single family.

Our family of Riemannian 3-manifolds arises from a one-parameter family of solvable Lie groups equipped with canonical left-invariant metrics. We denote the groups by $G_\alpha$ with $-1\leq\alpha\leq 1$. Each $G_\alpha$ is the semi-direct product of $\mathbb{R}$ with $\mathbb{R}^2$, with the following operation on $\mathbb{R}^3$:
$$(x,y,z)\ast(x',y',z')=(x'e^z+x, y'e^{-\alpha z}+y, z'+z).$$
Then $\mathbb{R}^3$, which is only the underlying set, can be equipped with the following left-invariant metric:
$$ds^2=e^{-2z}dx^2+e^{2\alpha z}dy^2+dz^2.$$
These $G_\alpha$ groups, when endowed with the canonical metric, perform our desired interpolation: linking the familiar and the unfamiliar. It is a natural question to analyze what happens in between.

We would also like to note the physical relevance of the $G_\alpha$ groups. First, the connection between the Bianchi groups (introduced first in \cite{LB}) and cosmology is well-explored. Indeed, a three-dimensional Lie group associated with a certain Bianchi Lie algebra corresponds to a symmetry group of space (as a part of space-time). The groups of Type VI (corresponding in our paper to $G_\alpha$ groups with $0<|\alpha|<1$) are associated to a cosmological model of space as homogeneous but anisotropic. Although the current consensus is that the universe is isotropic, one recent report \cite{WW} has experimental evidence suggesting the possibility of anisotropy. The authors of \cite{WW} report that the fine-structure constant exhibits statistically significant spatial variability. For further reading about the relationship between Bianchi groups and cosmology, see \cite{EM}. Lastly, the physical ramifications of the geodesic flow in a Lie group have been studied in detail since at least the work of \cite{VA}, where it was linked to hydrodynamics.

Our primary focus will be on $G_\alpha$ where $\alpha$ is positive. On this side of the family, we have the presence of both positive and negative sectional curvature, which makes the geometry of the geodesics and geodesic spheres quite interesting. In the limiting case when $\alpha=1$ (Sol) much is already known. Many properties of the geodesics were discovered by Grayson in \cite{G}. In \cite{MS}, Richard Schwartz and the author give an exact characterization of which geodesic segments in Sol are length minimizers, thereby giving a precise description of the cut locus of the identity in Sol (studying geodesics that start at the origin is sufficient since the spaces are homogeneous). This, in turn, led to the proof (also in \cite{MS}) that the geodesic spheres in Sol are homeomorphic to $S^2$. Recently, Richard Schwartz has computed area bounds of the spheres in Sol, and this interesting result may be found in \cite{S}. We look to generalize some of this work.

For positive $\alpha$, it happens that typical geodesics starting at the origin in $G_\alpha$ spiral around certain cylinders, as we will prove in Theorem 3.1. For each such geodesic, there is an associated period that determines how long it takes for it to spiral exactly once around. We denote the function determining the period by $P$, and we will show that it is a function of the initial tangent vector. We call a geodesic segment $\gamma$ of length $T$ $\textit{small, perfect,}$ or $large$ whenever $T<P_\gamma,$  $T=P_\gamma,$ or $T>P_\gamma,$ respectively. Our primary aim is the following conjecture:

\begin{conjecture*}
  A geodesic segment in $G_\alpha$ is a length minimizer
  if and only if it is small or perfect.
\end{conjecture*}

The above conjecture is already known for $G_1$, or Sol, and was proven in \cite{MS}. In this article, we will reduce the general conjecture to obtaining bounds on the derivative of the period function by proving the \textbf{Bounding Box Theorem}. In particular, we will define a certain curve $\partial_0 N$ in Section 3.3, and a key step in our would-be proof of the main conjecture is that a portion of this curve should be the graph of a monotonically decreasing function. The main obstacle to proving the whole conjecture is this monotonicity result. We are able to show the desired monotonicity of $\partial_0 N$ for the group $G_{1/2}$ with our \textbf{Monotonicity Theorem} because we have found an explicit formula for the period function in this case. The group $G_{1/2}$ maximizes scalar curvature in our family, which offers more credence that it is truly a "special case" along with Sol. Thus, our main theorem is a proof of the conjecture for $\alpha=1/2$:
\begin{theorem*}
A geodesic segment in $G_{1/2}$ is a length minimizer
  if and only if it is small or perfect.
\end{theorem*}

Here we outline our paper. In the next section, we will present the basic differential geometric facts about all of the $G_\alpha$ groups, ending with a look at the geodesic flow (restricted to the unit tangent bundle). In section 3, we begin by generalizing certain results from \cite{G} and \cite{MS} to all of the $G_\alpha$ with $0<\alpha\leq 1$. Among these are Theorem 3.1, various propositions from \cite{MS}, and Corollary 3.12, which proves half of our main conjecture: large geodesic segments are not length minimizing in any positive $\alpha$ group. Essential to our analysis is the idea of \textit{concatenation}, introduced first in \cite{MS}, that extends to the other $G_\alpha$. Just as in \cite{MS}, we afterwards turn our attention to using \textit{symmetric flow lines} to analyze the cut locus of each $G_\alpha$. Near the end of section 3, we state the Bounding Box Theorem (valid for all $0<\alpha\leq 1$) and the Monotonicity Theorem (which we only manage to prove when $\alpha=1/2$). Assuming these, we finish the proof of our main theorem. In section 4, we prove the Bounding Box Theorem, and in Section 5 we prove the Monotonicity Theorem by narrowing our analysis to the group $G_{1/2}$.  We have sufficient information about the period function for this group to extend the main result from \cite{MS}, obtaining a characterization of the cut locus in $G_{1/2}$ and its consequences for geodesic spheres. We also have two appendices: Appendix A contains the (tiresome) proof of a lemma used in Section 5 and Appendix B has the Mathematica code that generates the figures in our paper. We remark that many of the results we will prove can be considered independently of their geometric consequences, as properties of certain nonlinear ordinary differential equations. In particular, the proof of the Bounding Box Theorem is purely analytic in nature.

We would like to thank Richard Schwartz for his support throughout the development of this article. His encouragement was essential in finishing this work and his geometric insights were invaluable. We also thank him for pointing out omissions in earlier drafts. We thank Benoit Pausader for his advice on tackling the differential equations that we encountered and Georgios Daskalopoulos for teaching us differential geometry. We are also grateful to Stephen Miller for helping us with a numerical computation. We would like to thank the anonymous referees for their very helpful and detailed comments. Lastly, we would like to acknowledge Matthew Grayson's incisive work on Sol, which continues to inspire us.
\section{The Basic Structure}

In this section, we collect some basic facts about all of the $G_\alpha$ groups. The main idea is that a fruitful way of analyzing the geometry of these Lie groups is to first understand the geodesic flow, and this is the setting in which we will continue our analysis for the remainder of this paper.

The principal object of our study will be a one-parameter family of three-dimensional Lie groups, whose Lie algebras are of Type VI in Bianchi's classification, as elaborated in \cite{LB}. For ease of computation, we can also construct our groups as certain subgroups of $GL_3(\mathbb{R})$, and to that end we let
$$G_\alpha=\bigg\{ \begin{pmatrix} e^z&0&x \\ 0&e^{-\alpha z}&y\\0&0&1\end{pmatrix}\bigg| x,y,z\in\mathbb{R}\bigg\} \textrm{ for all } -1\leq\alpha\leq1$$
We can consider each $G_\alpha$ as a matrix group or, equivalently, as $\mathbb{R}^3$ with the following group law:
$$(x,y,z)\ast(x',y',z')=(x'e^z+x, y'e^{-\alpha z}+y, z'+z).$$
Then, $\mathbb{R}^3$ with this group law has the following left-invariant metric
$$ds^2=e^{-2z}dx^2+e^{2\alpha z}dy^2+dz^2.$$
The Lie algebra of $G_\alpha,$ which we denote $\frak{g}_\alpha$, has the following orthonormal basis:
$$\bigg\{ X=\begin{pmatrix} 0&0&1 \\ 0&0&0\\0&0&0\end{pmatrix}\quad  Y=\begin{pmatrix} 0&0&0 \\ 0&0&1\\0&0&0\end{pmatrix}\quad  Z=\begin{pmatrix} 1&0&0 \\ 0&-\alpha&0\\0&0&0\end{pmatrix}\bigg\}$$
or
\begin{equation}X=e^z \frac{\partial}{\partial x}, Y=e^{-\alpha z}\frac{\partial}{\partial y}, Z=\frac{\partial}{\partial z}.\end{equation}
So, the structure equations are:
\begin{equation}[X,Y]=0\quad [Y,Z]=\alpha Y\quad [X,Z]=-X\end{equation}
The behavior of $\frak{g}_\alpha$ coincides with the Lie algebras of Type VI (in Bianchi's Classification) when $0<|\alpha|<1$ and we have three limiting cases: $\alpha=1,$ which is a Bianchi group of type VI$_{0}$, $\alpha=0,$ which is a Bianchi group of type III, and  $\alpha=-1$, which is a Bianchi group of type V. The intermediate cases are not unimodular, unlike the limiting cases, and that may also be of interest. As a consequence of the classification done in \cite{LB}, no two of the Lie algebras $\frak{g}_\alpha$ are isomorphic, hence
\begin{proposition}[Bianchi, \cite{LB}]
No two of the $G_\alpha$ are Lie group isomorphic.
\end{proposition}
One might ask what occurs if we let the parameter $|\alpha|>1$, and the answer is simple. In this case, the Lie algebra will be isomorphic to that of one of our $G_\alpha$ groups with $|\alpha|<1$. Since we restrict ourselves to simply connected Lie groups (indeed, Lie groups that are  diffeomorphic to $\mathbb{R}^3$), this means that the corresponding Lie groups are isomorphic as well. 

Now we recall an essential fact from Riemannian geometry. Given a smooth manifold with a Riemannian metric, there exists a unique torsion-free connection which is compatible with the metric, which is called the Levi-Civita connection. This can be easily proven with the following lemma, the proof of which may be found in \cite{K}.
\begin{lemma}[Koszul Formula]
Let $\nabla$ be a torsion-free, metric connection on a Riemannian manifold $(M,g)$. Then, for any vector fields $X,Y,$ and $Z,$ we have:
$$2g(\nabla_X Y, Z)= X(g(Y,Z))+Y(g(X,Z))-Z(g(X,Y))+$$
$$+g([X,Y],Z)-g([X,Z],Y)-g([Y,Z],X)$$
\end{lemma}
Using the Koszul Formula, we can easily compute the Levi-Civita connection $\nabla$ for each $G_\alpha$ from equation $(2)$. We get:
\begin{proposition}
The Levi-Civita connection of $G_\alpha$, with its left-invariant metric, is completely determined by
$$\begin{pmatrix} \nabla_X X && \nabla_X Y&&\nabla_X Z\\ \nabla_Y X && \nabla_Y Y&&\nabla_Y Z\\ \nabla_Z X && \nabla_Z Y&&\nabla_Z Z\end{pmatrix}=\begin{pmatrix} Z && 0&&-X\\ 0 && -\alpha Z&&\alpha Y\\ 0 && 0&&0\end{pmatrix}$$
where $\{X,Y,Z\}$ is the orthonormal basis of the Lie algebra, as in $(1)$.
\end{proposition}
\begin{proof}
As an illustration of this computation, we derive $\nabla_X X$ explicitly. The other entries in the matrix of covariant derivatives can be found in an identical manner. First, since $[X,X]=0$ and $X,Y,Z$ form an orthonormal basis of the Lie Algebra, the Koszul formula yields:
$$2g(\nabla_X X, V)=-2g([X,V],X),$$
where $V\in \{X,Y,Z\}$. Hence, using the structure equations in $(2)$, we get
$$g(\nabla_X X, X)=0, \quad g(\nabla_X X, Y)=0, \quad g(\nabla_X X, Z)=Z$$
so $\nabla_X X= Z,$ as desired.
\end{proof}
The coordinate planes play a special role in the geometry of $G_\alpha.$ Each group (which is diffeomorphic to $\mathbb{R}^3$) has three foliations by the $XZ,$ $YZ,$ and $XY$ planes. It will be worthwhile to compute the curvatures of these surfaces in $G_\alpha$. The sectional curvature can be computed easily from the Levi-Civita Connection, while the extrinsic (Gaussian) curvature and mean curvature are computed using the Weingarten equation. Lastly, for surfaces in a Riemannian 3-manifold we have the relation 
\begin{equation} Intrinsic = Extrinsic + Sectional \end{equation}
from Gauss' Theorema Egregium. For details, see the first chapter in \cite{N}. Straightforward computations yield the following proposition and its immediate consequences.
\begin{proposition} The relevant curvatures of the coordinate planes:
\begin{center}
 \begin{tabular}{|c c c c c|} 
 \hline
Plane& Sectional & Intrinsic & Extrinsic (Gaussian) & Mean  \\ 
 \hline\hline
XY& $\alpha$ & 0 & $-\alpha$ & $(1-\alpha)/2$ \\ 
 \hline
XZ& -1 & -1 & 0 &0 \\
 \hline
YZ& $-\alpha^2$ & $-\alpha^2$ & 0 &0\\ 
 \hline
\end{tabular}
\end{center}
\end{proposition}
\begin{corollary}
The $XY$ plane is a minimal surface (having vanishing mean curvature) in $G_\alpha$ if and only if $\alpha=1$, and the $XZ$ and $YZ$ planes are minimal for all $\alpha$. Also, the $XY$ plane is always a constant-mean-curvature surface.
\end{corollary}
We will later strengthen this corollary by proving that the $XZ$ and $YZ$ planes are geodesically embedded.

It can be easily seen that $G_1$ is a model for the Sol geometry, $G_0$ is a model of $\mathbb{H}^2\cross\mathbb{R}$, and $G_{-1}$ is a model of $\mathbb{H}^3,$ or hyperbolic space. We can also compute the Ricci and scalar curvatures as well and we get that the scalar curvature of $G_\alpha$ is $S_\alpha=2\alpha-2-2\alpha^2$. $S_\alpha$ attains its maximum over the family at $S_{1/2}=-3/2$, and the minimum is attained at $S_{-1}=-6$ (as might be expected). Therefore, the group $G_{1/2}$ may also be considered a special case: the member of the interpolation that maximizes scalar curvature. Moreover, $S_\alpha$ is symmetric in the positive side of the family, in the sense that for all non-negative $\alpha$, $S_{\alpha}=S_{1-\alpha}$. 

Other self-evident properties of the coordinate planes could be stated, but we let the reader find these.
Now, we turn our attention to the geodesic flow of $G_\alpha$. Rather than attempting to derive analytic formulas for the geodesics from the geodesic equation, as done for Sol in \cite{T}, we restrict the geodesic flow to the unit-tangent bundle and consider the resulting vector field. The idea of restricting the geodesic flow to $S(G_1)$ for Sol was first explored by Grayson in his thesis \cite{G} and then used by Richard Schwartz and the current author to characterize the cut locus of the origin of Sol in \cite{MS}. We recall that the cut locus of a point $p$ in a Riemannian manifold $(M,g)$ is the locus of points on geodesics starting at $p$ where the geodesics cease to be length minimizing. 

Now, we extend the previous ideas to the other $G_\alpha$ groups. Indeed, consider $\frak{g}_\alpha$ and let $S(G_\alpha)$ be the unit sphere centered at the origin in $\mathfrak{g}_\alpha.$ Suppose that $\gamma(t)$ is a geodesic parametrized by arc length such that $\gamma(0)$ is the identity of $G$. Then, we can realize the development of $\gamma'(t)$, the tangent vector field along $\gamma$, as a curve on $S(G_\alpha),$ which will be the integral curve of a vector field on $S(G_\alpha)$ denoted by $\Sigma_\alpha.$ We compute the vector field $\Sigma_\alpha$ explicitly. 
\begin{proposition}
For the group $G_\alpha$ the vector field $\Sigma_\alpha$ is given by
$$\Sigma_{\alpha}(x,y,z)= (xz, -\alpha yz, \alpha y^2-x^2)$$
\end{proposition}
\begin{proof}
Since we are dealing with a homogeneous space (a Lie group) it suffices to examine the infinitesimal change of $V=\gamma'(0)=(x,y,z).$ We remark that parallel translation along $\gamma$ preserves $\gamma'$ because we have a geodesic, but parallel translation does not preserve the constant (w.r.t. the left-invariant orthonormal frame) vector field $V=\gamma'(0)$ along $\gamma$. Indeed, the infinitesimal change in the constant vector field $V$ as we parallel translate along $\gamma$ is precisely the covariant derivative of $V$ with respect to itself, or $\nabla_V V.$ Then, our vector field on $S(G_\alpha)$ is precisely:
$$\Sigma_\alpha=\nabla_V (\gamma' - V)=\nabla_V \gamma' -\nabla_V V=-\nabla_V V.$$
We also remark that since $V$ is a constant vector field, $\Sigma_\alpha$ is determined completely by the Levi-Civita connection that we previously computed, and this computation is elementary.
\end{proof}
We remark where the equilibria points of $\Sigma_\alpha$ are, since these correspond to straight-line geodesics. When $0<\alpha\leq 1,$ the equilibria are
$$\bigg(\pm\sqrt{\frac{\alpha}{1+\alpha}},\pm\sqrt{\frac{1}{1+\alpha}},0\bigg) \textrm{ and } (0,0,\pm1).$$
When $\alpha=0$, the set $\{X=0\}\cap S(G_\alpha)$ is an equator of equilibria, and when $\alpha<0,$ the only equilibria are at the poles.
A glance at $\Sigma_\alpha$ gets us our promised strengthening of Corollary 2.5:
\begin{corollary}
The $XZ$ and $YZ$ planes are geodesically embedded. The $XY$ is never geodesically embedded, even when it is a minimal surface (i.e. for $\alpha=1$).
\end{corollary}
\begin{proof}
Let $\gamma$ be a geodesic starting at the origin, with initial tangent vector in the $XZ$ plane of $\frak{g}_\alpha$ (i.e. the $Y$ coordinate of $\gamma'(0)$ is $0$). Then, $\Sigma_\alpha$ tells us that the $Y$ coordinate of $\gamma'(t)$ is $0$ $\forall t>0$. Hence $XZ$ is totally geodesic. The proof for the $YZ$ plane is identical. Now we prove the statement about the $XY$ plane. Observe that for any initial direction (excluding the straight line geodesic) in the $XY$ plane, the third coordinate of $\Sigma_\alpha$ is different from $0$. Thus, a geodesic $\gamma$ with a direction in the $XY$ plane cannot stay in the $XY$ plane, proving the last desired assertion.
\end{proof}
Consider the complement of the union of the two planes $X=0$ and $Y=0$ in $\frak{g}_\alpha$. This is the union of four connected components, which we call \textit{sectors}. Since the $XZ$ and $YZ$ planes are geodesically embedded, we have
\begin{corollary}
The Riemannian exponential map, which we denote by $E$, preserves each sector of $\frak{g}_\alpha$. In particular, if $(x,y,z)\in \frak{g}_\alpha$ is such that $x,y>0,$ then $E(x,y,z)=(a,b,c)$ with $a,b>0$. 
\end{corollary}
We will use Corollary 2.8 often and without mentioning it. We make the following key observation about $\Sigma_\alpha$.
\begin{proposition}
The integral curves of $\Sigma_\alpha$ are precisely the level sets of the function $H(x,y,z)=\abs{x}^\alpha y$ on the unit sphere. 
\end{proposition}
\begin{proof}
Without loss of generality, we consider the positive sector. We recall that the symplectic gradient is the analogue in symplectic geometry of the gradient in Riemannian geometry. In the case of the sphere with standard symplectic structure, the symplectic gradient is defined by taking the gradient of the function $H$ (on the sphere) and rotating it 90 degrees counterclockwise. Doing this computation for $H,$ yields $\nabla_{sym}H=x^{\alpha-1}\cdot\Sigma_\alpha(x,y,z).$ Since this vector field is the same as the structure field up to a scalar function, the desired property follows. 
\end{proof}
\begin{remark}
$\Sigma_\alpha$ is a Hamiltonian system in these coordinates if and only if $\alpha=1,$ i.e. for Sol ($G_1$) 
\end{remark}
We finish this section with a conjecture that, if true, provides some connection between the groups in the $G_\alpha$ family, for $\alpha \in [-1,1]$. We recall that the \textit{volume entropy}, $h$, of a homogeneous Riemannian manifold $(M,g)$ is a measure of the volume growth in $M$. We can define
$$h(M,g):=\lim_{R\rightarrow \infty} \frac{\log(\textrm{Vol } B(R))}{R}$$
where $B(R)$ is a geodesic ball of radius $R$ in $M$. Since $G_{-1}$ is a model of Hyperbolic space and $G_{1}$ is Sol, we know that $h(G_{-1})=2$ and $h(G_{1})=1$ (see \cite{S}). Based on this, we conjecture that:
\begin{conjecture}
$h(G_\alpha)$ is a monotonically decreasing function of $\alpha$ for $\alpha \in [-1,1]$.
\end{conjecture}
\section{The Positive Alpha Family}
In this section, we start to explore the positive $\alpha$ side of the family. These geometries exhibit common behaviors such as geodesics always lying on certain cylinders, spiraling around in a "periodic-drift" manner. A natural way to classify vectors in the Lie algebra is by how much the associated geodesic segment under the Riemannian exponential map spirals around its associated cylinder. This classification allows us to discern how the exponential map behaves with great detail.
\subsection{Grayson Cylinders and Period Functions}
More than a few of our theorems in Section 3 may be considered generalizations of results in \cite{G} and \cite{MS}. To begin our analysis, we study the \textit{Grayson Cylinders} of the $G_\alpha$ groups. 
\begin{definition}
We call the level sets of $H(x,y,z)=|x|^\alpha y$ that are closed curves \textbf{loop level sets}.
\end{definition}
We prove the following theorem by adapting a method first used in \cite{G} for Sol.
\begin{theorem}[The Grayson Cylinder Theorem]
Any geodesic with initial tangent vector on the same loop level set as 
$$\bigg(\beta\sqrt{\frac{\alpha}{1+\alpha}}, \frac{\beta}{\sqrt{1+\alpha}}, \sqrt{1-\beta^2}\bigg), \beta\in[0,1]$$
lies on the cylinder given by 
$$w^2+e^{2z}+\frac{1}{\alpha}e^{-2\alpha z}=\frac{1+\alpha}{\alpha}\cdot\frac{1}{\beta^2}$$
where $w=x-y\sqrt{\alpha}$. We call these cylinders "Grayson Cylinders".
\end{theorem}
\begin{proof}
We will consider, without loss of generality, the positive sector. Additionally, since we can always travel along a geodesic until we reach an initial tangent vector as in the statement above, it suffices to prove that such geodesics lie, locally, on our cylinders. More precisely, the above path of initial tangent vectors is a path on the unit tangent sphere which intersects each level set exactly once (a one-parameter family of level sets). Projecting the vector field $\Sigma_\alpha$ onto the $xy$ plane, we get \textit{flat flow lines}, which correspond to the following flows on the unit tangent sphere:
$$S_\beta(t)=\bigg(\beta\sqrt{\frac{\alpha}{1+\alpha}} e^{t}, \frac{\beta}{\sqrt{1+\alpha}} e^{-\alpha t}, z_\beta(t)\bigg)$$
where 
$$z_\beta(t)=\sqrt{1-\frac{\beta^2}{1+\alpha}(\alpha e^{2t}+e^{-2\alpha t})}.$$
The function $z_\beta(t)$ is only defined until the flat flow lines reach the unit circle; however, we are only concerned with the local behavior of the geodesics. The parametrization of the level sets induced from the \textit{flat} flow lines is not actually the one induced by $\Sigma_\alpha,$ indeed, there is a scalar function $\sigma_\beta(t)$ such that $S_\beta'(t)=\sigma_\beta(t)\cdot\Sigma_\alpha(S_\beta(t)).$ After a calculation, we see that $\sigma_\beta(t)=1/z_\beta(t).$ Now, if $\gamma(t)$ is the point on the geodesic corresponding to $S_\beta(t),$ it follows that 
$$\frac{d\gamma(t)}{dt}=\sigma_\beta(t)S_\beta(t),$$
whence the $z$ coordinate of $\gamma(t)$ is linear in $t.$ The $x,y$ travel is similarly obtained, after accounting for their respective infinitesimal distortion, which is done using the orthonormal basis of the Lie algebra as presented in equation $(1)$ . We have:
$$\gamma_x(t)=\beta\sqrt{\frac{\alpha}{1+\alpha}}\int_0^te^{2t}\sigma_\beta(t)dt$$
and
$$\gamma_y(t)=\frac{\beta}{\sqrt{1+\alpha}}\int_0^te^{-2\alpha t}\sigma_\beta(t)dt.$$
The function $w(t)=\gamma_x(t)-\sqrt{\alpha}\gamma_y(t)$ can be explicitly integrated and doing so gets us (where $z=t$)
$$w^2+e^{2z}+\frac{1}{\alpha}e^{-2\alpha z}=\frac{1+\alpha}{\alpha}\cdot\frac{1}{\beta^2}$$
as desired.
\end{proof}
If we consider Grayson Cylinders as regular surfaces in $\mathbb{R}^3$ with the ordinary Euclidean metric, then a simple derivation of their first and second fundamental forms reveals that they are surfaces with Gaussian curvature identically equal to zero. Hence, they are locally isometric to ordinary cylinders, and, because they are also diffeomorphic to ordinary cylinders, Grayson Cylinders are in fact isometric to ordinary cylinders, for all choices of $\alpha$ and $\beta$.

It is easier to gauge the shape of a Grayson Cylinder by looking at its projection onto the planes normal to the line $x-y\sqrt{\alpha}$, or, alternatively, as the implicit plot of a function of the two variables $w$ and $z$ defined in the statement of Theorem 3.1. It appears that as $\alpha$ is fixed, the Grayson Cylinders limit to two "hyperbolic slabs" as $\beta$ goes to zero. Alternatively, as $\beta$ is fixed and $\alpha$ varies, it appears that one side of the Grayson Cylinder is ballooning outwards. In Figures 1 and 2, we have some examples generated with the Mathematica code provided in Section 7.1.
\begin{figure}[H]
\centering
\begin{subfigure}{0.45\textwidth}
\includegraphics[width=\textwidth]{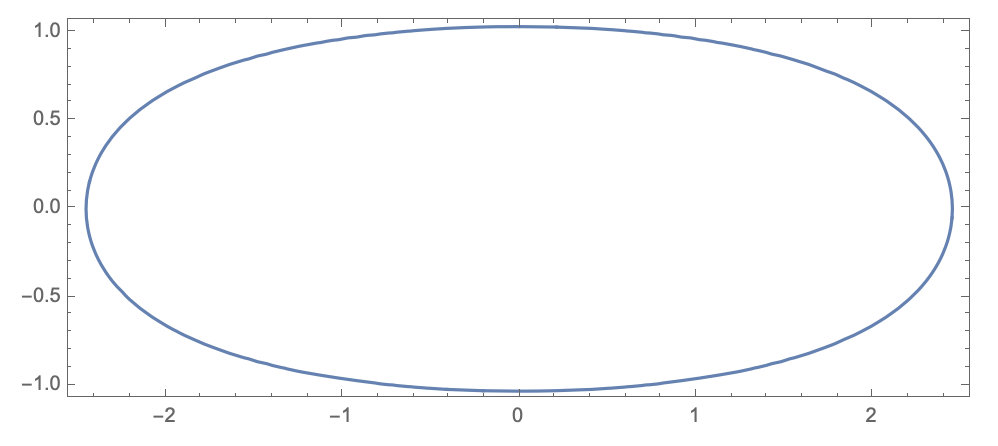}
\caption{$\alpha=1$ and $\beta=1/2$.}
\end{subfigure}
\hfill
\begin{subfigure}{0.45\textwidth}
\includegraphics[width=\textwidth]{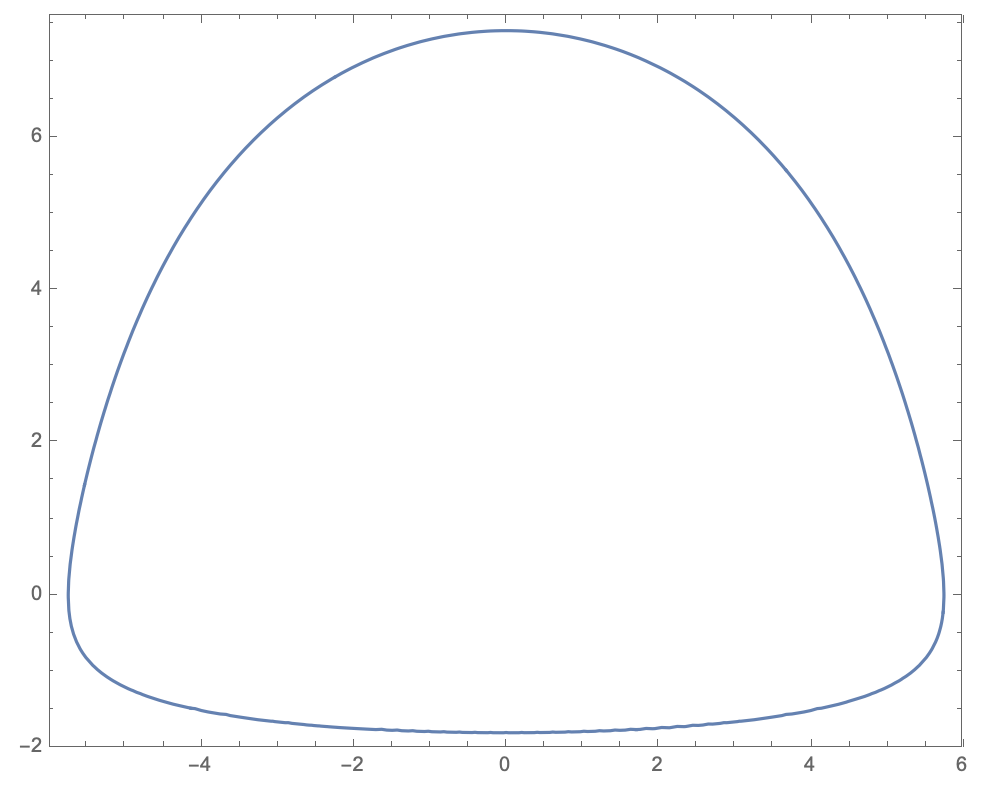}
\caption{$\alpha=1/4$ and $\beta=1/2$.}
\end{subfigure}
\caption{Slices of Grayson Cylinders with varying $\alpha$}
\end{figure}
\begin{figure}[H]
\centering
\begin{subfigure}{0.45\textwidth}
\includegraphics[width=\textwidth]{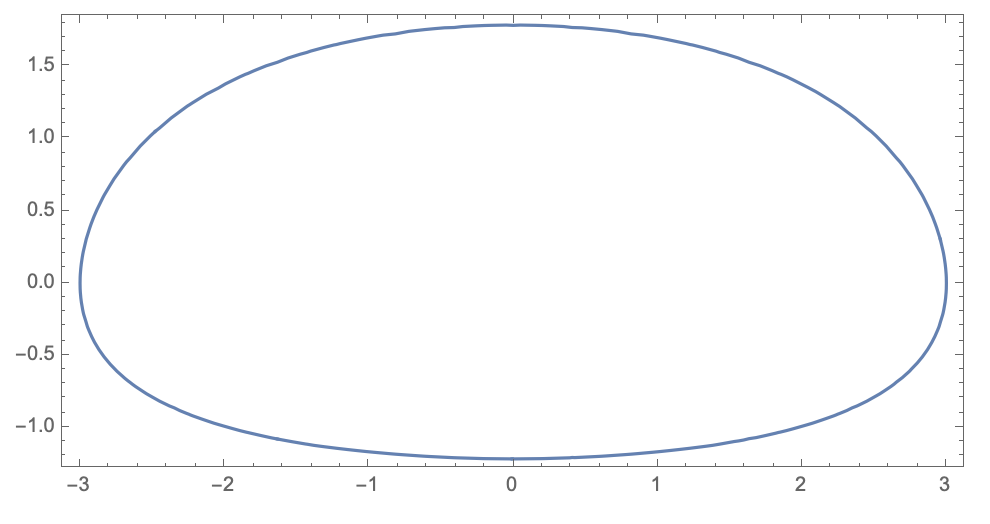}
\caption{$\alpha=1/2$ and $\beta=1/2$.}
\end{subfigure}
\hfill
\begin{subfigure}{0.45\textwidth}
\includegraphics[width=\textwidth]{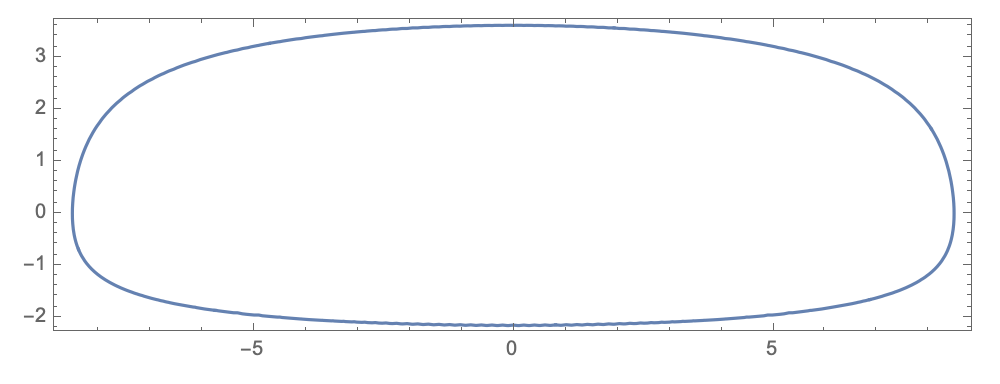}
\caption{$\alpha=1/2$ and $\beta=1/5$.}
\end{subfigure}
\caption{Slices of Grayson Cylinders with varying $\beta$}
\end{figure}
We denote loop level sets by $\lambda$. Each loop level set has an associated period, $P_\lambda,$ which is the time it takes for a flowline to go exactly once around $\lambda$, and it suffices to compute the period at one vector in a loop level set to know $P_\lambda$. We can compare $P_\lambda$ to the length $T$ of a geodesic segment $\gamma$ associated to a flowline that starts at some point of $\lambda$ and flows for time $T.$ We call $\gamma$ $\textit{small, perfect,}$ or $large$ whenever $T<P_\lambda,$  $T=P_\lambda,$ or $T>P_\lambda,$ respectively. It seems that this "classification" of geodesic segments in $G_\alpha$ is ideal. For instance, it was shown in \cite{MS} that a geodesic in Sol is length-minimizing if and only if it is small or perfect. We now derive an integral formula for $P_\lambda$ and simplify the integral in two special cases. 
\begin{proposition}
Let $\lambda$ be the loop level set associated to the vector 
$$V_\beta=\bigg(\beta\sqrt{\frac{\alpha}{1+\alpha}}, \frac{\beta}{\sqrt{1+\alpha}}, \sqrt{1-\beta^2}\bigg),$$
then 
$$P_\lambda(\beta) = \int_{-t_1}^{t_0} \frac{2dt}{\sqrt{1-\frac{\beta^2}{\alpha+1}(\alpha e^{2t}+e^{-2\alpha t})}}$$
where $t_0$ and $t_1$ are the times it takes to flow from $V_\beta$ to the equator of $S(G_\alpha)$ in the direction of, and opposite to the flow of $\lambda$, respectively. 
\end{proposition}
\begin{proof}
The loop level sets are symmetric with respect to the $XY$ plane. Also, we recall from the proof of Theorem 3.1 that 
$$\frac{d\gamma(t)}{dt}=\sigma_\beta(t)S_\beta(t), \textrm{ hence } \bigg\|\frac{d\gamma(t)}{dt}\bigg\|=\sigma_\beta(t).$$
Thus, the length of a perfect geodesic segment $\gamma$ starting at $V_\beta$ is
$$P_\lambda(\beta)=2\cdot\textrm{Length}(\gamma)=2\cdot\int_{-t_1}^{t_0}\bigg\|\frac{d\gamma(t)}{dt}\bigg\|dt=\int_{-t_1}^{t_0} \frac{2dt}{\sqrt{1-\frac{\beta^2}{\alpha+1}(\alpha e^{2t}+e^{-2\alpha t})}}.$$
\end{proof}
\begin{remark}
The times $t_0$ and $t_1$ are precisely when the flat flow lines hit the unit circle, or when 
\begin{equation}\alpha e^{2t_0}+e^{-2\alpha t_0}=\frac{\alpha+1}{\beta^2} \textrm{ and } \alpha e^{-2t_1}+e^{2\alpha t_1}=\frac{\alpha+1}{\beta^2}\end{equation}
\end{remark}
Stephen Miller helped us to numerically compute the period function for any choice of positive $\alpha$. The Mathematica code for this can be found in Section 7.6. An explicit formula for the period function in Sol ($G_1$) was derived in \cite{MS} and \cite{T}. It is:
$$P_\lambda(\beta)=\frac{4}{\sqrt{1+\beta^2}}\cdot K\bigg(\frac{1-\beta^2}{1+\beta^2}\bigg)$$
where $K(m)$ is the complete elliptic integral of the first kind, with the parameter as in Mathematica.

A closed-form expression of $P_\lambda$ can also be obtained for $G_{1/2}$. Since elliptic integrals have been studied extensively and many of their properties are well-known, the following expression allows us to analyze $P_\lambda$ more easily. 
\begin{corollary}
When $\alpha=1/2,$ or for the group $G_{1/2}$, the period function is given by 
$$P_\lambda(\beta)=\frac{4\sqrt{3}}{\beta\sqrt{e^{t_0-t_1}+2e^{t_1}}}\cdot K\bigg(\frac{2(e^{t_1}-e^{-t_0})}{e^{t_0-t_1}+2e^{t_1}}\bigg)$$
where
$$t_0=\log\bigg(\frac{1}{\beta}\cdot\bigg(\frac{1}{(-\beta^3 + \sqrt{-1 + \beta^6})^{\frac{1}{3}}} + (-\beta^3 + 
       \sqrt{-1 + \beta^6})^{\frac{1}{3}}\bigg)\bigg)$$
        and
   $$t_1=\log\bigg(\frac{1}{2}\bigg(\frac{1}{\beta^2}+\frac{1}{\beta^4(-2+\frac{1}{\beta^6}+\frac{2\sqrt{-1+\beta^6}}{\beta^3})^{\frac{1}{3}}}+(-2+\frac{1}{\beta^6}+\frac{2\sqrt{-1+\beta^6}}{\beta^3})^{\frac{1}{3}}\bigg)\bigg)$$
\end{corollary}
\begin{proof}
The fact that $t_0$ and $t_1$ are as above is nothing more than solving Equation $(4)$, which becomes a cubic polynomial when $\alpha=\frac{1}{2}$. Now, from Proposition 3.2, we have the integral formula for the period, and we perform the change of variables $u=e^t$:
$$P_\lambda(\beta) = \int_{-t_1}^{t_0} \frac{2dt}{\sqrt{1-\frac{2\beta^2}{3}(\frac{1}{2} e^{2t}+e^{-t})}}=\int_{e^{-t_1}}^{e^{t_0}} \frac{2du}{\sqrt{u^2-\frac{2\beta^2}{3}(\frac{1}{2}u^4+u)}}.$$
Further simplifications yield:
$$P_\lambda(\beta)=\frac{2\sqrt{3}}{\beta}\cdot\int_{e^{-t_1}}^{e^{t_0}}\frac{du}{\sqrt{u(\frac{3}{\beta^2}u-u^3+2)}}.$$
However, we know the four roots of the quartic polynomial in the square root, they are $u=0, e^{-t_1}, e^{t_0}$, and using Vi\`{e}te's Formulas, $-2e^{t_1-t_0}$. So, we factor and get
$$P_\lambda(\beta)=\frac{2\sqrt{3}}{\beta}\cdot\int_{e^{-t_1}}^{e^{t_0}}\frac{du}{\sqrt{u(e^{t_0}-u)(u-e^{-t_1})(u+2e^{t_1-t_0})}}.$$
The above integral has already been computed for us in terms of elliptic integrals. Indeed, we find our integral in formula 6 of section 3.147, page 275 in the tome \cite{GR}. Further simplifications get us our desired expression for $P_\lambda$.
\end{proof}
Here we state an essential fact, which is forthrightly supplied to us by the above expression of $P_\lambda(\beta)$ in terms of an elliptic integral. We have:
\begin{proposition}
$$\frac{d}{d\beta}\big(P_\lambda(\beta)\big)<0$$
when $\alpha=1$ and $\alpha=1/2$. Moreover, for $\alpha=1$,
$$\lim_{\beta\to 1} P_\lambda(\beta)=\pi\sqrt{2}$$
and for $\alpha=1/2$,
$$\lim_{\beta\to 1} P_\lambda(\beta)=2\pi.$$
\end{proposition}
We could not find a similar formula for $P_\lambda$, when $\alpha$ is not $1$ or $1/2$, in terms of elliptic integrals or hypergeometric functions. This is unfortunate, as Proposition 3.5 is vital to the method here and in \cite{MS} to characterize the cut locus of $G_1$ and $G_{1/2}$. However, we can still numerically compute the period function as presented in the appendix, and this allows us to conjecture:
\begin{conjecture} Let $P(\beta)$ be the period function in $G_\alpha$, then
$$\frac{d}{d\beta}\big(P_\lambda(\beta)\big)<0$$
and
$$\lim_{\beta\to 1} P(\beta)=\frac{\pi\sqrt{2}}{\sqrt{\alpha}}.$$
\end{conjecture}
This conjecture would lend something quantitative to the idea that the bad behavior (or the cut locus) of $G_\alpha$ dissipates at infinity as Sol interpolates to $\mathbb{H}^2\cross\mathbb{R}$. With the Mathematica code in Section 7.6, we get the following numerical evidence for Conjecture 3.6.
\begin{center}
 \begin{tabular}{|c c c|} 
 \hline
$\alpha$ & Numerical Value of $P(\alpha,\beta=.999)$ & $\pi\sqrt{2}/\sqrt{\alpha}$  \\ 
 \hline\hline
 0.1 & 14.0792 & 14.0496 \\
  \hline
 0.2 & 9.94735 & 9.93459 \\
  \hline
 0.3 & 8.11985 & 8.11156 \\
  \hline
 0.4 & 7.03114 & 7.02481 \\
  \hline
 0.5 & 6.28842 & 6.28319 \\
  \hline
 0.6 & 5.7403 & 5.73574 \\
  \hline
 0.7 & 5.31436 & 5.31026 \\
  \hline
 0.8 & 4.97106 & 4.96729 \\
  \hline
 0.9 & 4.68673 & 4.68321 \\
  \hline
 1. & 4.44622 & 4.44288 \\
  \hline
\end{tabular}
\end{center}

\subsection{Concatenation and Some Other Useful Facts}
An important property which extends from Sol to $G_\alpha$ for $0<\alpha<1$ is that the loop level sets are symmetric with respect to the plane $Z=0.$ This simple observation allows the technique of $\it{concatenation},$ essential to the analysis in \cite{MS}, to be replicated for all of the $G_\alpha$ groups, when $\alpha>0$. For the interested reader, Richard Schwartz's Java program \cite{S2} uses concatenation to generate geodesics and geodesic spheres in Sol, and a modified version of this program can generate the spheres and geodesics in any $G_\alpha$ group as well as in other Lie groups, such as Nil. As an illustration of the power of this technique in numerical simulations, we present in Figure 3 the geodesic spheres of radius around $5$ in four different geometries. The spheres are presented from the same angle, the purple line is the $z$ axis, and the red lines are the horizontal axes. The salient phenomenon is that one "lobe" of the sphere is contracting as $\alpha$ goes to zero. Qualitatively, this corresponds to the dissipation of the "bad" behavior (or the cut locus) of $G_\alpha$ as $\alpha$ tends to zero, since the amount of shear diminishes.
\begin{figure}[H]
\centering
\begin{subfigure}{0.45\textwidth}
\includegraphics[width=0.9\linewidth]{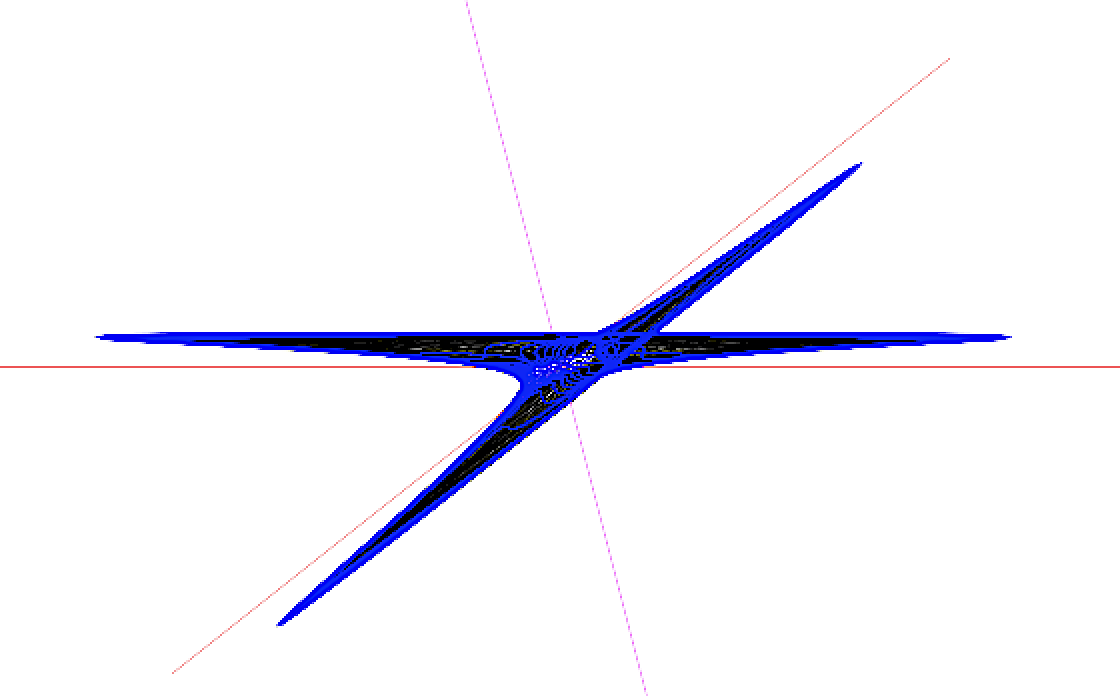}
\caption{In Sol, or the group $G_1$.}
\end{subfigure}
\hfill
\begin{subfigure}{0.45\textwidth}
\includegraphics[width=0.9\linewidth]{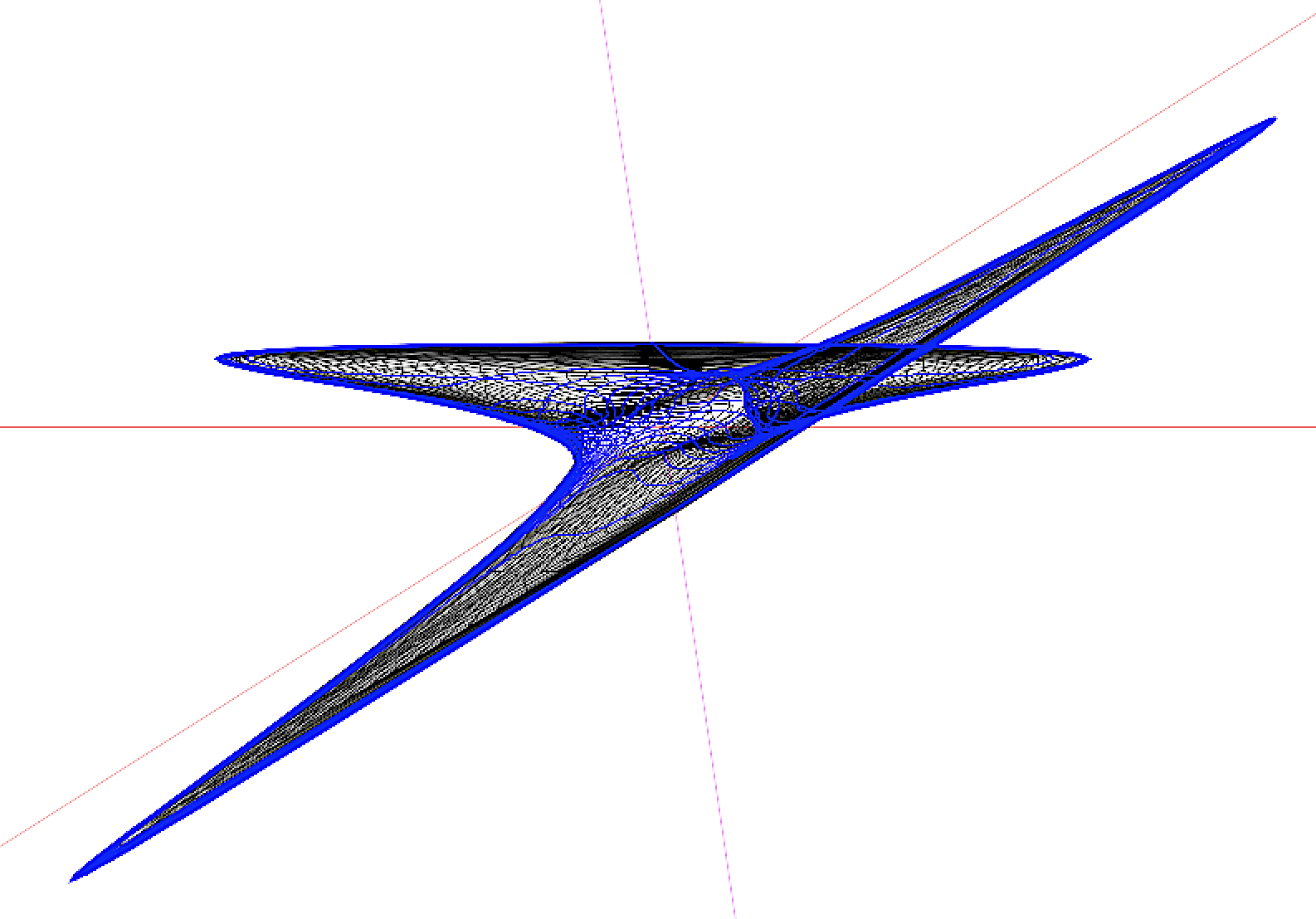}
\caption{In the group $G_{3/4}$.}
\end{subfigure}
\begin{subfigure}{0.45\textwidth}
\includegraphics[width=0.9\linewidth]{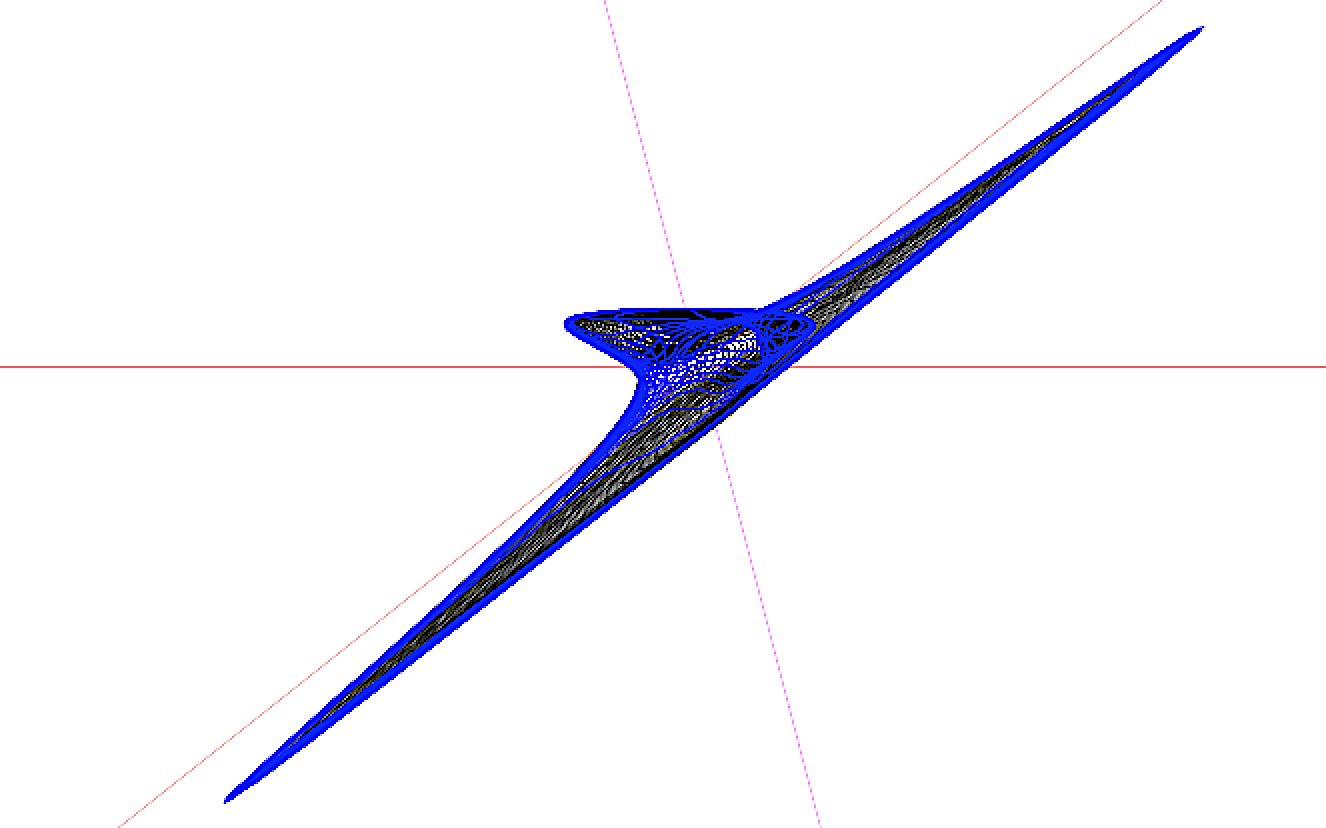}
\caption{In the group $G_{1/2}$}
\end{subfigure}
\begin{subfigure}{0.45\textwidth}
\includegraphics[width=0.9\linewidth]{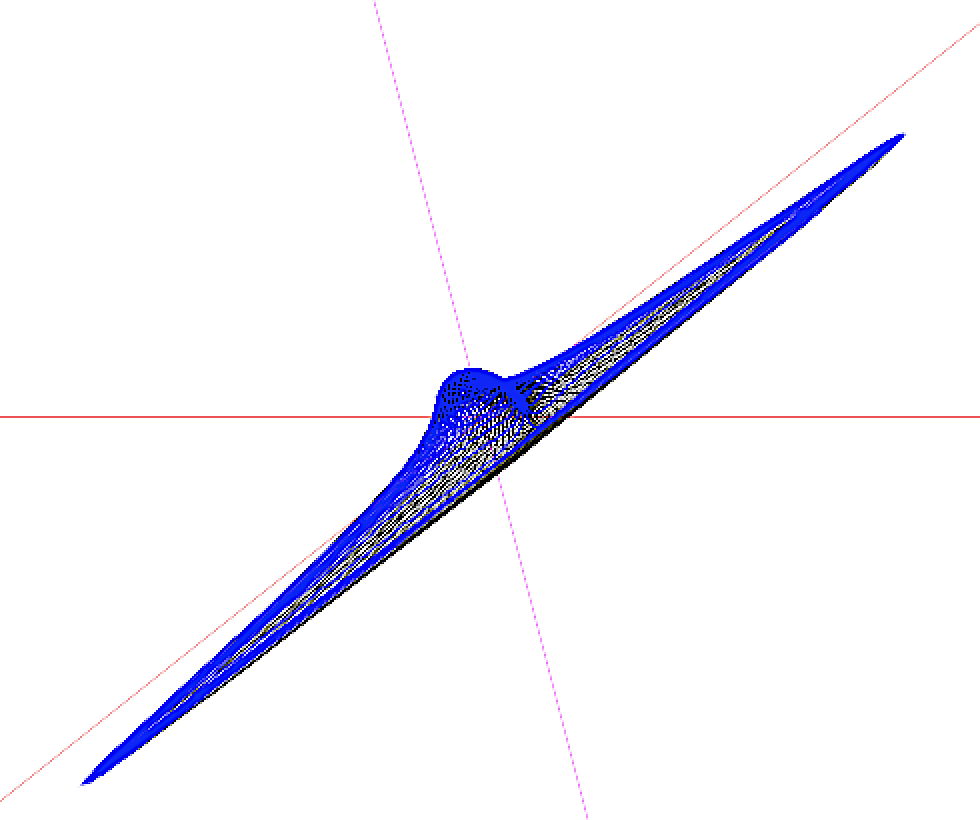}
\caption{In the group $G_0$, or $\mathbb{H}^2\cross \mathbb{R}$.}
\end{subfigure}
\caption{Geodesic spheres of radius around $5$ in four different geometries.}
\end{figure}
 Each flowline $\lambda$ of our vector field corresponds to a segment of a geodesic $\gamma$. Let $T$ be the time it takes to trace out $\lambda,$ then $T$ is exactly the length of $\gamma$ since we always take unit-speed geodesics. Let $L_\lambda$ be the far endpoint of $\gamma$ and consider the equally spaced times
$$0=t_0<t_1<\ldots <t_n=T$$
with corresponding points $\lambda_0, \ldots, \lambda_n$ along $\lambda.$ Then we have
$$L_\lambda=\lim_{n\to\infty} (\epsilon_n \lambda_0)\ast\ldots\ast(\epsilon_n \lambda_n),\quad \epsilon_n=T/(n+1)$$
where $\ast$ is the group law in $G_\alpha.$ The above equation is well-defined because the underlying space of both $G_\alpha$ and its Lie algebra is $\mathbb{R}^3.$ We will also use the notation $\lambda=a|b$ to indicate that we are splitting $\lambda$ into two sub-trajectories, $a$ and $b.$ 

The above equation yields:
$$L_\lambda=L_a \ast L_b$$
when $\lambda=a|b.$
We set $\epsilon_n \lambda_j=(x_{n,j},y_{n,j},z_{n,j})$.
Vertical displacements commute in $G_\alpha,$ therefore the third coordinate of
the far endpoint of $\gamma$ is given by
$$\lim_{n \to \infty} \sum_{j=0}^n z_{n,j}.$$
From this, and the symmetry of the flow lines with respect to $Z=0,$ we get the following lemmas. Since these results are an immediate extension of the results in \cite{MS}, we provide only sketches of their proofs.
\begin{lemma}
If the map $(x,y,z) \to (x,y,-z)$ exchanges the
two endpoints of the flowline $\lambda$ then the endpoints of the geodesic segment $\gamma$
both lie in the plane $Z=0$ and $L_\lambda$ is a
horizontal translation. In this case we call
$\lambda$ {\it symmetric\/}.
\end{lemma}
\begin{proof}
Since $\lambda$ is symmetric, the sum $\lim_{n \to \infty} \sum_{j=0}^n z_{n,j}$ vanishes, so the total vertical displacement is zero.
\end{proof}
\begin{lemma}
If $\lambda$ is not symmetric then we can write
  $\lambda=a|b|c$ where $a,c$ are either symmetric
  or empty, and $b$ lies entirely above or entirely
  below the plane $Z=0$.  Since $L_a$ and
  $L_c$ are horizontal translations -- or just
  the identity in the empty cases -- and
    $L_b$ is not such a translation, the endpoints of $\lambda$ are
  not in the same horizontal plane.
  \end{lemma}
 \begin{lemma} If $\lambda=a|b$, where both $a$ and $b$ are
symmetric, then both endpoints of
$\gamma$ lie in the plane $Z=0$.  We can do this whenever
$\lambda$ is one full period of a loop
level set.  Hence, a perfect
geodesic segment has both endpoints in the same
horizontal plane.
\end{lemma}
\begin{proof}
Since $\lambda=a|b$, we know that $L_\lambda=L_a\ast L_b$. Since $L_a$ and $L_b$ are horizontal translations by Lemma 3.7, it follows that $L_\lambda$ stays in the same $Z=0$ plane.
\end{proof}
\begin{lemma} If $\lambda_1$ and $\lambda_2$ are full trajectories of the same
  loop level set, then
we can write $\lambda_1=a|b$ and $\lambda_2=b|a$, which leads to
$L_{\lambda_2}=L_a^{-1}L_{\lambda_1}L_a$. Working this out with the group law in $G_\alpha$ gives:
$(a_1 e^{-z}, b_1 e^{\alpha z}, 0)=(a_2, b_2, 0)$, where $(x,y,z)=L_a$ and $(a_i, b_i, 0)=L_{\lambda_i}$.
In particular, we have $a_1^\alpha b_1=a_2^\alpha b_2$ and we call the function (of the flowlines)
$H_\lambda=\sqrt{|a_1^\alpha b_1|}$ the {\it holonomy invariant\/} of
the loop level set $\lambda.$ 
\end{lemma}

Let $E$ be the Riemannian
exponential map.
We call $V_+=(x,y,z)$ and $V_-=(x,y,-z)$, vectors in the Lie algebra, {\it partners\/}.
The symmetric trajectories discussed in Lemma 3.7 have endpoints
which are partners.
Note that if $V_+$ and $V_-$ are partners, then 
one is perfect if and only if the other one is, because they lie on the same loop level set. The next two facts are generalizations of results from \cite{MS} about Sol, and, in particular, Corollary 3.12 proves half of our main conjecture.

\begin{theorem}
If $V_+$ and $V_-$ are perfect partners, then
$E(V_+)=E(V_-)$.
\end{theorem}
\begin{proof}
Let $\lambda_{\pm}$ be the trajectory which
makes one circuit around the loop level set starting
at $U_{\pm}$.  As above, we write the flowlines as
$\lambda_+=a|b$ and
$\lambda_-=b|a$.  Since $V_+$ and $V_-$ are partners, we can take
$a$ and $b$ both to be symmetric.
But then the elements
$L_a$, $L_b$, $L_{\lambda_1}$, $L_{\lambda_2}$
all preserve the plane $Z=0$ and hence mutually commute, by Lemma 3.7.
By Lemma 3.10, we have
$L_{\lambda_+}=L_{\lambda_-}$.
But $E(V_{\pm})=L_{\lambda_{\pm}}$.
\end{proof}

\begin{corollary}
A large geodesic segment is not a length minimizer.
\end{corollary}
\begin{proof}
If this is false then, by shortening our geodesic, we
can find a perfect geodesic segment $\gamma$, corresponding
to a perfect vector $V=(x,y,z)$, which is a unique
geodesic minimizer without conjugate points.  If $z \not =0$
we immediately contradict Theorem 3.11. 
If $z=0$, we consider the variation,
$\epsilon \to \gamma(\epsilon)$ through same-length
perfect geodesic segments
$\gamma(\epsilon)$ corresponding to
the vector $V_{\epsilon}=(x_{\epsilon},y_{\epsilon},\epsilon)$.
The vectors $V_{\epsilon}$ and $V_{-\epsilon}$ are partners, so
$\gamma(\epsilon)$ and $\gamma(-\epsilon)$
have the same endpoint. Hence, this variation
corresponds to a conjugate point on $\gamma$
and again we have a contradiction.
\end{proof}
The next step is to analyze what happens for small and perfect geodesics. To begin, we point out another consequence of Theorem 3.11. Let $M$ be the set of vectors in the Lie algebra of $G_\alpha$ associated to small geodesic segments and let $\partial M$ be the set of vectors associated to perfect geodesic segments. Lastly, let $\partial_0 M$ be the intersection of $\partial M$ with the plane $Z=0.$ Since $E$ identifies perfect partner vectors, we have a vanishing Jacobi field at each point of $\partial_0 M.$ However, we still have:
\begin{proposition}
$dE$ is nonsingular in $\partial M - \partial_0 M$ 
\end{proposition}
\begin{proof}
Let $V\in \partial M - \partial_0 M$ and let $\gamma$ be its corresponding geodesic segment. Let $S$ be the round sphere around the origin in the Lie algebra that contains $V.$ By the Gauss Lemma, we know that it suffices to show $dE$ is nonsingular when restricted to the tangent plane of $S$ at $V.$ We produce two linearly independent geodesic variations that proves the sufficient condition.

First, let $V_t$ be a curve on $S$ perpendicular to all loop level sets and such that the associated geodesic segments $\gamma_t$ are small for $t>0.$ In the proof of the Grayson Cylinder Theorem, we showed that the third-coordinate of our geodesic segment will vary linearly in $t$ away from $Z=0.$ This is our first nonzero variation. Now consider the variation $V_t$ on the same loop level set. Since all $V_t$ are perfect, with the same holonomy invariant (as in Lemma 3.10), the actual holonomy of $\gamma_t$ will vary on the hyperbola $x^\alpha y= H_\lambda^2,$ which is our second, linearly independent, nonzero variation. We note that this proposition is equivalent to stating that no vector in $\partial M -\partial_0 M$ has a conjugate point, i.e. the geodesic segment associated to any vector in $\partial M -\partial_0 M$ cannot have any non-trivial Jacobi Field along it which vanishes at its endpoints.
\end{proof}
Another useful consequence of what we have heretofore shown is that the Holonomy function (defined in Lemma 3.10) is monotonically increasing. The following result will be useful later, when we analyze the behavior of $E$ on the set of perfect vectors. More precisely, we have:
\begin{proposition}
Let $P$ be the unique period associated to a flowline $\lambda$. We know that the holonomy $H_\lambda$ is an invariant of the flowline, so it is a function of $P$. Moreover, $H_\lambda$ varies monotonically with the flowlines. Explicitly, we have $\frac{dH}{dP}(P)>0$.
\end{proposition}
\begin{proof}
Consider the point (that is on the straight line geodesic) $(R\sqrt{\alpha} , R, 0)$ in $G_\alpha$, where $R>0$ is a large number. Let $V$ be the shortest vector in the Lie algebra of $G_\alpha$ such that $E(V)=(R\sqrt{\alpha}, R, 0)$. Such a vector is guaranteed to exist because the straight-line geodesic (which, however, may not be length-minimizing) sends a vector in the Lie algebra to this point. By Corollary 3.12, $V$ is either a short or a perfect vector. In any case, there exists a constant $\delta\geq 1$ such that $\delta V$ is a perfect vector, and since $\delta V$ is perfect, we also know that $E(\delta V) = (a,b,0)$ for some $a,b$ such that $a\geq R\sqrt{\alpha}$ and $b\geq R$. It follows that 
$$H(P(\delta V))=H(\| \delta V\|) \geq R^{\frac{1+\alpha}{2}}\alpha^{\alpha/4}.$$
Since $\alpha$ is a constant for each $G_\alpha$ and $R$ can be arbitrarily large, we see that $H$ is an unbounded function of $P$. 

Now we show that $dH/dP$ cannot vanish, which coupled with the unboundedness of $H$, proves the desired monotonicity result. Assume that $dH/dP(P_0)=0$ for some choice of period $P_0$ (in other words, for some choice of flowline $\lambda_0$ of period $P_0$). Choosing the necessary $\nu$, we consider the vector 
$$U_0=\bigg(\nu\sqrt{\frac{\alpha}{1+\alpha}}, \frac{\nu}{\sqrt{1+\alpha}}, \sqrt{1-\nu^2}\bigg)=\big(x(0)\sqrt{\alpha},x(0),z(0)\big)$$
that is on the flowline $\lambda_0$. Theorem 3.1 states that the geodesic with this initial tangent vector spirals around its corresponding Grayson Cylinder with companion line $x=y\sqrt{\alpha}$ in the $Z=0$ plane of Sol. Thus, its endpoint after spiraling exactly once around is in fact $E(U_0)=(a\sqrt{\alpha},a,0)$ for some number $a$. We consider the variation of perfect vectors in the positive sector $U(t)=(x(t)\sqrt{\alpha}, x(t), z(t))$ such that $\|U(t)\|=P_0 +t$. Since we assumed that $dH/dP(P_0)=0$, it follows that $da/dt(0)=0$ which implies that $dE$ is singular at the perfect vector $U_0$, a contradiction. 
\end{proof}
To begin our analysis of the small geodesic segments, we prove an interesting generalization of the \textit{Reciprocity Lemma} from \cite{MS}. If $V$ is a perfect vector, then $E(V)$ will lie in the $Z=0$ plane by Lemma 3.9; however, we can get something better.
\begin{theorem}
Let $V=(x,y,z)$ be a perfect vector. There exists a number $\mu\neq 0$ such that $E(V)=\mu(\alpha y,x,0).$
\end{theorem}
\begin{proof}
We proceed as in \cite{MS}. As usual, it suffices to prove the result for the positive sector. Let $g(t)=(x(t),y(t),z(t))$ be a flowline of $\Sigma_\alpha$ with initial conditions 
$$(x(0),y(0),z(0))=\bigg(\nu\sqrt{\frac{\alpha}{1+\alpha}}, \frac{\nu}{\sqrt{1+\alpha}}, \sqrt{1-\nu^2}\bigg),$$
where $\nu$ is a constant corresponding to the appropriate level set, and let $(a(t),b(t),0)=E(x(t),y(t),z(t))$. We define the functions
$$h(t)=\frac{1}{\alpha}\frac{a(t)}{b(t)}\quad \textrm{and}\quad v(t)=\frac{y(t)}{x(t)}.$$
Where it is understood, we shall avoid writing that the functions defined above are functions of $t.$ To prove the theorem, it suffices to show $h(t)=v(t)$ for all $t,$ which we do by demonstrating that $h$ and $v$ are solutions of the same ODE initial value problem. It is evident that $v(0)=1/\sqrt{\alpha},$ and we can see that $h(0)=1/\sqrt{\alpha}$ by geometric considerations. We recall from Theorem 3.1 that a geodesic with initial velocity vector $(x(0), y(0), z(0))$ spirals with the companion line (in the $Z=0$ plane) $x-y\sqrt{\alpha}$. Since we have a perfect vector, the holonomy $(a(0),b(0),0)$ lies on this line, whence: $h(0)=\frac{1}{\alpha}\frac{a(0)}{b(0)}=\frac{1}{\alpha}\sqrt{\alpha}=\frac{1}{\sqrt{\alpha}}=v(0).$
Having shown that $h(0)=v(0),$ all that remains is to show that the functions satisfy the same ODE. By using Lemma 3.10, we can consider the following approximation of the infinitesimal change in $(a(t),b(t),0):$
$$(a(t+\epsilon),b(t+\epsilon),0)=u_\epsilon^{-1} \ast (a(t),b(t),0)\ast u_\epsilon \textrm{ where } u_\epsilon=(\epsilon x, \epsilon y, \epsilon z),$$
or
$$(a(t+\epsilon),b(t+\epsilon),0)=(a(t)e^{-\epsilon z}, b(t)e^{\alpha\epsilon z},0).$$
Then, we have
$$\frac{d}{dt}h(t)=\frac{1}{\alpha}\lim_{\epsilon\rightarrow 0}\bigg(\frac{a(t+\epsilon)}{b(t+\epsilon)}-\frac{a(t)}{b(t)}\bigg)$$
so
$$\frac{d}{dt}h(t)=\frac{1}{\alpha}\frac{a}{b}\lim_{\epsilon\rightarrow 0}\frac{e^{-\epsilon(1+\alpha)z}-1}{\epsilon}=-(1+\alpha)zh.$$
We can use the structure field and elementary calculus to compute the other derivative:
$$\frac{d}{dt}v(t)=\frac{d}{dt}(y/x)=\frac{xy'-yx'}{x^2}=-(1+\alpha)zv.$$
Having shown that $v,h$ satisfy the same initial value problem, it follows that they are equal for all time. Another, shorter proof of this lemma will be offered later.
\end{proof}
\subsection{Symmetric Flowlines}
We now introduce another technique introduced first in \cite{MS}: the emphasis on symmetric flow lines, which is justified by Lemma 3.9. We introduce the following sets in $\frak{g}_\alpha$ and $G_\alpha$:
\begin{itemize}
\item Let $M, \partial M \subset \frak{g}_\alpha$ be the set of small and perfect vectors, as previously defined.

\item Let $\Pi$ be the $XY$ plane in $\frak{g}_\alpha$ and $\tilde{\Pi}$ be the $XY$ plane in $G_\alpha$.

\item Let $\partial_0 M=\partial M \cap \Pi$.

\item Let $M^{symm} \subset M$ be those small vectors which correspond to symmetric flowlines as in Lemma $3.7$.

\item Let $\partial_0 N=E(\partial_0 M)$.

\item Let $\partial N$ be the complement, in $\tilde{\Pi}$, of the
  component of $\tilde{\Pi}-\partial_0 N$ that contains the origin.
  
\item Let $N=G_\alpha-\partial N$.

\item For any set $A$ in either $\frak{g}_\alpha$ or $G_\alpha,$ we denote $A_+$ to be the elements of $A$ in the positive sector, where $x,y>0$.
\end{itemize}
Our underlying goal is to show that $\partial N$ is the cut locus of the origin in $G_\alpha$. This has already been done for Sol ($G_1$) in \cite{MS}, and we shall prove the same for $G_{1/2}$. Thus, although the notation we use here is suggestive of certain topological relationships (e.g. is $\partial N$ the topological boundary of $N$?), we are only able to prove these relationships for $G_{1/2}$ in the present paper. Reflections across the $XZ$ and $YZ$ planes are isometries in every $G_\alpha$, so proving something for the positive sector (where $x,y>0$) proves the same result for every sector. This is useful in simplifying many proofs. 

A first step towards proving that the cut locus is $\partial N$ is to show that $$E(M) \cap \partial N =\emptyset,$$ or, intuitively, that the exponential map "separates" small and perfect vectors. The following lemma is a step towards this.

\begin{lemma}
If $E(M) \cap \partial N \not = \emptyset$, then
$E(M_+^{symm}) \cap \partial N_+ \not = \emptyset$.
\end{lemma}
\begin{proof}
Let $V=(x,y,z) \in M$ be such that $E(V) \in \partial N$. By symmetry, it suffices to assume that $x,y \geq 0$.
Since $E$ is sector-preserving, we must have
$E(V) \in \partial N_+$. If $x=0$ then $E(V)$ lies in the plane $X=0$, a set which
is disjoint from $\partial N_+$.  Hence $x>0$.  Similarly, $y>0$. Since $V\in M,$ $V$ is associated
to a small flowline.  Since $\partial N_+ \subset \tilde{\Pi}$,
we must have $E(V) \in \tilde{\Pi}$.  Then, $V$ is associated
to a small symmetric flowline, by Lemma 3.9.
In this case, we must have $z\neq0$ because the endpoints
of small symmetric flowlines are partner points in
the sense of Theorem 3.11. So,
$V \in M_+^{{\rm symm}}$, as claimed.
\end{proof}
With this lemma in hand, we should analyze the symmetric flowlines in detail in order to prove that $E(M_+^{symm}) \cap \partial N_+= \emptyset$.  Symmetric flowlines are governed by a certain system of nonlinear ordinary differential equations.  Let $\Theta_P^+$ denote those points in the
(unique in the positive sector) loop level set of period $P$ having
all coordinates positive. Every element
of $M_+^{{\rm symm\/}}$ corresponds to a small symmetric
flowline starting in $\Theta_P^+$.
\newline
\newline
{\bf The Canonical Parametrization:\/}
The set $\Theta_P^+$ is an open arc. 
We fix a period $P$ and we set $\rho=P/2$.
Let $p_0=(x(0),y(0),0) \in \Theta_P \cap \Pi$ be the point with $x(0)>y(0)$. The initial value $x(0)$ varies from $\sqrt{(\alpha+1)/\alpha}$ to $1$.
We then let 
\begin{equation}
p_t=(x(t),y(t),z(t))
\end{equation}
be the point on $\Theta_P^+$ which we reach after time $t \in (0,\rho)$ by 
flowing {\it backwards\/} along the structure field
$\Sigma$.   That is
\begin{equation}
\label{backwards}
\frac{dp}{dt}=
(x',y',z')=-\Sigma(x,y,z)=(-xz,+\alpha yz,x^2-\alpha y^2).
\end{equation}
Henceforth, we use the notation $x'$ to stand for $dx/dt$, etc.
\newline
\newline
{\bf The Associated Flowlines:\/} 
We let $\hat{p}_t$ be the partner of $p_t$, namely
\begin{equation}
\hat{p}_t=(x(t),y(t),-z(t)).
\end{equation}
 We let $\lambda_t$ be the small symmetric
flowline having endpoints $p_t$ and $\hat{p}_t$.
Since the structure field $\Sigma$
points downward at $p_0$, the symmetric flowline $\lambda_t$
starts out small and increases all the way to a perfect
flowline as $t$ increases from $0$ to $\rho$.
We call the limiting perfect flowline $\lambda_{\rho}$.
\newline
\newline
\noindent
{\bf The Associated Plane Curves:\/}
Let $V_t \in M_+^{{\rm symm\/}}$ be the vector
corresponding to $\lambda_t$.  (Recall that $E(V_t)=L_{\lambda_t}$)
Define
\begin{equation}
\Lambda_P(t):=E(V_t)=(a(t),b(t),0) \hskip 30 pt t \in (0,\rho].
\end{equation}
These plane curves are in $\tilde{\Pi}$ because they are endpoints of symmetric flowlines, and they will be among our main objects of interest in what follows. In Figure 4, we present a collection of the plane curves (colored blue) for $\alpha=1/2$ with the choice of $x(0)$ varying from $0.6$ to $0.95$ at intervals of $0.05$. We also include the initial value $x_0=1/\sqrt{3}$, which corresponds to the straight geodesic segment in $G_{1/2}$. The black curve is an approximation of $\partial_0 N_+$, or endpoints of perfect flowlines, which are the right-hand endpoints of each $\Lambda_P$ curve.
\begin{figure}[h!]
\centering
\includegraphics[width=1\textwidth]{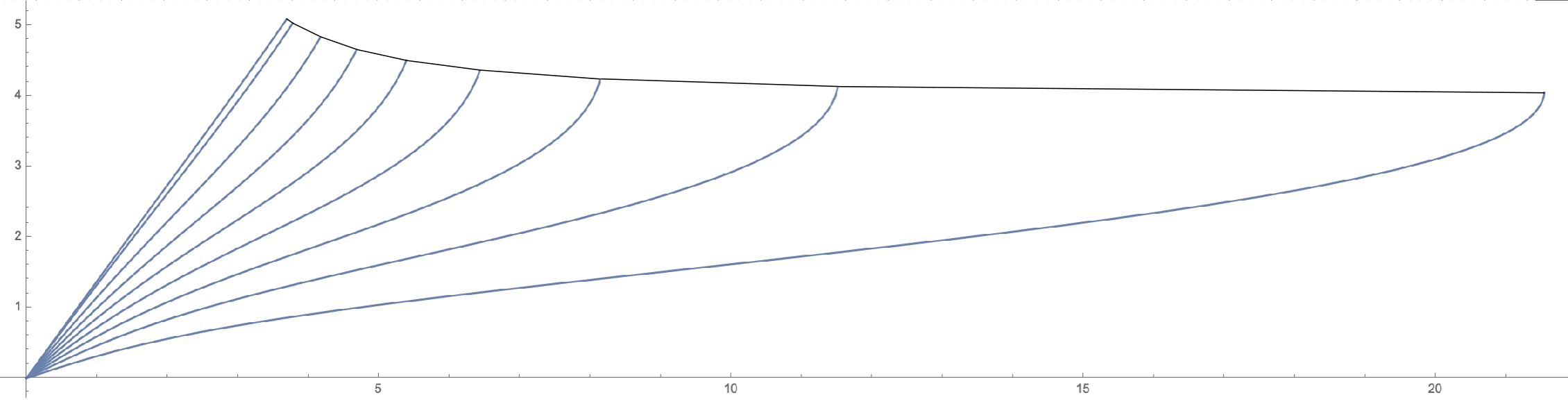}
\caption{The image of $\Lambda_P$ over the interval $(0,\rho]$ for varying $x_0$ and $\partial_0 N_+$. }
\end{figure}
\begin{lemma}
\label{endpoint}
  $\Lambda_P(\rho) \in \partial_0 N_+$, and $0<b(\rho)<a(\rho)$.
\end{lemma}
\begin{proof}
We have $\Lambda_P(\rho) \in \partial_0 N_+$ because $\lambda_{\rho}$ is perfect and
starts at $(x(\rho),y(\rho),0)$.
Note that, by symmetry of the flowlines,
\begin{equation} x(0)^\alpha = y(\rho) \textrm{ and } y(0) = x(\rho)^\alpha\end{equation}
Hence
$x(\rho)<y(\rho)$. Theorem 3.15, applied to
the perfect vector $V_{\rho}$, now gives 
$$0<b(\rho)<a(\rho).$$
\end{proof}

We have $E(M_+^{{\rm symm\/}}) \cap \partial N_+=\emptyset$ provided that
\begin{equation}
\label{goal}
\Lambda_P(0,\rho) \cap  \partial N_+=\emptyset,
\textrm{ for all periods } P.
\end{equation}
So all we have to do is establish Equation \ref{goal}. Let $B_P$ be the rectangle in the $XY$ plane with vertices 
$$(0,0,0), (0,b(\rho),0), (a(\rho),0,0), \textrm{ and } (a(\rho),b(\rho),0).$$ 
Our first step in proving Equation \ref{goal} is to contain the image of $\Lambda_P$ with the following theorem, which we will prove to be true for each $G_\alpha$ group:
\begin{theorem}[The Bounding Box Theorem]
$\Lambda_P(0,\rho) \subset {\rm interior\/}(B_P)$ for all $P$.
\end{theorem}
The Bounding Triangle Theorem serves a similar role in \cite{MS} for Sol $(G_1)$, but it cannot be generalized to any other $G_{\alpha}$ group. It states that $\Lambda_P(0,\rho)$ is contained inside the  triangle with vertices $(0,0,0), (a(\rho),0,0)$, and $(a(\rho),b(\rho),0)$. In Figure 6, we depict the image of a single plane curve $\Lambda_P$ for $\alpha=1/2$ and $x_0=0.99945$, which illustrates the failure of the Bounding Triangle Theorem in the other Lie groups.

Now, if we could also manage to show ${\rm interior\/}(B_P)\cap \partial N_+ =\emptyset$, we would finish proving Equation \ref{goal}. Since the Bounding Box Theorem is not as powerful as the Bounding Triangle theorem of \cite{MS}, we need more information about $\partial N_0$ to prove Equation 10 than was needed in \cite{MS}. We succeed in performing this second step for the group $G_{1/2}$ by getting bounds on the derivative of the period function (using its expression in terms of an elliptic integral in that case). The necessary ingredient that we get is
\begin{theorem*}[The Monotonicity Theorem]
For $\alpha=1/2$, $\partial_0 N_+$ is the graph of a non-increasing function (in Cartesian coordinates).
\end{theorem*}
\subsection{Proof of the Main Results for $G_{1/2}$}
For $G_{1/2}$, assuming that the Bounding Box and Monotonicity Theorems are true, we can proceed to characterize the cut locus of the identity.

First, we prove equation $(10)$:
\begin{theorem}
For the group $G_{1/2}$ we have, for all $P$,
$$\Lambda_P(0,\rho) \cap \partial N_+ = \emptyset$$
\end{theorem}
\begin{proof}
By the Bounding Box Theorem we know that $\Lambda_P(0,\rho) \subset {\rm interior\/}(B_P)$ for all $P$. By the Monotonicity Theorem, we know that $\partial_0 N_+$ is the graph of a decreasing function in Cartesian coordinates, so we conclude that $\partial N_+$ is disjoint from ${\rm interior\/}(B_P)$ for all $P$. Our desired result holds. 
\end{proof}
The above theorem, combined with Lemma 3.16 gets us:
\begin{corollary}
\label{smallperfect} 
$$E(M) \cap \partial N=\emptyset$$
\end{corollary}
The rest of our argument for showing that the cut locus of $G_{1/2}$ is $\partial N$ follows exactly as in \cite{MS}.
Let $E$ be Riemannian exponential map.
Let $\cal M$ be the component
of $\partial M_+-\partial_0M_+$ which contains
vectors with all coordinates positive.
Let ${\cal N\/}=\partial N_+ - \partial_0 N_+$. We first prove a few lemmas. 
\begin{lemma}
The map $E$ is injective on $\mathcal{M}$.
\end{lemma}
\begin{proof}
Let $V_1$ and $V_2$ be two vectors in $\mathcal{M}$ such that $E(V_1)=E(V_2)$. We also let $U_1=E(V_1)$ and $U_2=E(V_2)$ and denote the $j^{th}$ coordinate of $U_i$ as $U_{ij}$ and likewise for $\frac{V_i}{\|V_i\|}$. Since $U_1$ and $U_2$ have the same holonomy invariant and since the holonomy is monotonic with respect to choice of flowline (Proposition 3.14), it follows that $\frac{V_1}{\|V_1\|}$ and $\frac{V_2}{\|V_2\|}$ lie on the same loop level set in $S(G_{1/2})$. Thus, $V_{11}V_{12}^2=V_{21}V_{22}^2$. By the Reciprocity Lemma, and since $U_1=U_2$, we get 
$$\frac{V_{12}}{V_{11}}=\frac{V_{22}}{V_{21}}.$$ 
We can now conclude that $V_{11}=V_{21}$ and $V_{12}=V_{22}$. Since $\|V_1\|=\|V_2\|$, we get $V_1=V_2$, finishing the proof.
\end{proof}
\begin{lemma}
  \label{SP0}
  $E({\cal M\/}) \subset \cal N$.
\end{lemma}
\begin{proof}
The map $E$ is injective
on ${\cal M\/} \cup \partial_0 M_+$, by the previous lemma. At the same
time, $E(\partial_0 M_+)=\partial_0 N_+$.
Hence
\begin{equation}
  \label{alternative}
  E({\cal M\/}) \subset \Pi - \partial_0 N_+.
\end{equation}
By definition, $\cal N$ is one of the components of the
$\Pi-\partial_0 N_+$.  Therefore,
since $\cal M$ is connected, the image
$E({\cal M\/})$ is either contained in $\cal N$ or disjoint from $\cal N$.
Since the sets are evidently not disjoint (large perfect vectors land far away from the identity and near the line $x=y/\sqrt{2}$), we have containment. 
\end{proof}

\begin{corollary}
\label{SP}
$E(\partial M) \cap E(M)=\emptyset$.
\end{corollary}
\begin{proof}
Up to symmetry, every vector in $\partial M$ lies either in
$\cal M$ or in $\partial_0 M_+$.
By definition, $E(\partial_0 M)=\partial_0 N \subset \partial N$.  So,
by the previous result, we have
$E(\partial M) \subset \partial N$.
By Corollary \ref{smallperfect} we have
$E(M) \cap \partial N=\emptyset$.
Combining these two statements gives the result.
\end{proof}

\begin{theorem}
  \label{minimi}
  Perfect geodesic segments are length minimizing.
\end{theorem}
\begin{proof}
Suppose $V_1 \in \partial M$ and
$E(V_1)=E(V_2)$ for some $V_2$ with $\|V_2\|< \|V_1\|$.
By symmetries of $G_{1/2}$ and the flowlines, we can assume that both $V_1$ and $V_2$ are in the positive sector and that their third coordinates are also positive.
By Corollary 3.12, we have $V_2 \in M \cup \partial M$.
By Corollary \ref{SP} we have $V_2 \in \mathcal{M}$.
But then $V_1=V_2$, by Lemma 3.21, which contradicts $\|V_2\|< \|V_1\|$.
\end{proof}

The results above identify $\partial N$ as the cut locus of the
identity of $G_{1/2}$ just as obtained in \cite{MS} for Sol. We can summarize by saying 
\begin{theorem}
A geodesic segment in $G_{1/2}$ is a length
minimizer if and only if it is small or perfect.
\end{theorem}
In addition, small geodesic segments are
unique length minimizers and they have no conjugate points.
Hence, using standard results about the cut locus, as in \cite{N}, we get that $E: M \to N$
is an injective, proper, local diffeomorphism. This implies
that $E: M \to N$ is also surjective and hence a diffeomorphism.
Moreover, $E: \partial_0 M_+ \to \partial_0 N_+$ is a diffeomorphism, by similar considerations. Results about the geodesic spheres in $G_{1/2}$ follow immediately, as in \cite{MS} for Sol, by "sewing up" $\partial M$ in a 2-1 fashion with $E$.
In particular, we have:
\begin{corollary}
Geodesic spheres in $G_{1/2}$ are always topological spheres.
\end{corollary} 

The rest of our paper is devoted to proving our technical results: the Bounding Box Theorem and the Monotonicity Theorem. We will prove the Bounding Box Theorem in full generality, i.e. for all $\alpha\in (0,1]$. However, we only manage to prove the Monotonicity Theorem for $G_{1/2}$, where we have an expression of the period in terms of an elliptic integral. It is our opinion that either an expression for $P$ in terms of hypergeometric functions exists for general $\alpha$ or a thorough analysis of the (novel?) integral function in Proposition 3.2 can be done to demonstrate the monotonicity results required. Regardless, our Bounding Box Theorem does half of the work necessary to finish the proof of our main conjecture: for all $G_\alpha$ groups, a geodesic segment is length minimizing if and only if it is small or perfect. 

We reiterate that the necessary step to prove our conjecture is to show the Monotonicity Theorem holds for general $G_\alpha$ and that there is encouraging numerical evidence supporting this proposition. We plan to investigate this last step and prove our main conjecture in the future.

\begin{figure}[H]
\centering
\includegraphics[width=1\textwidth]{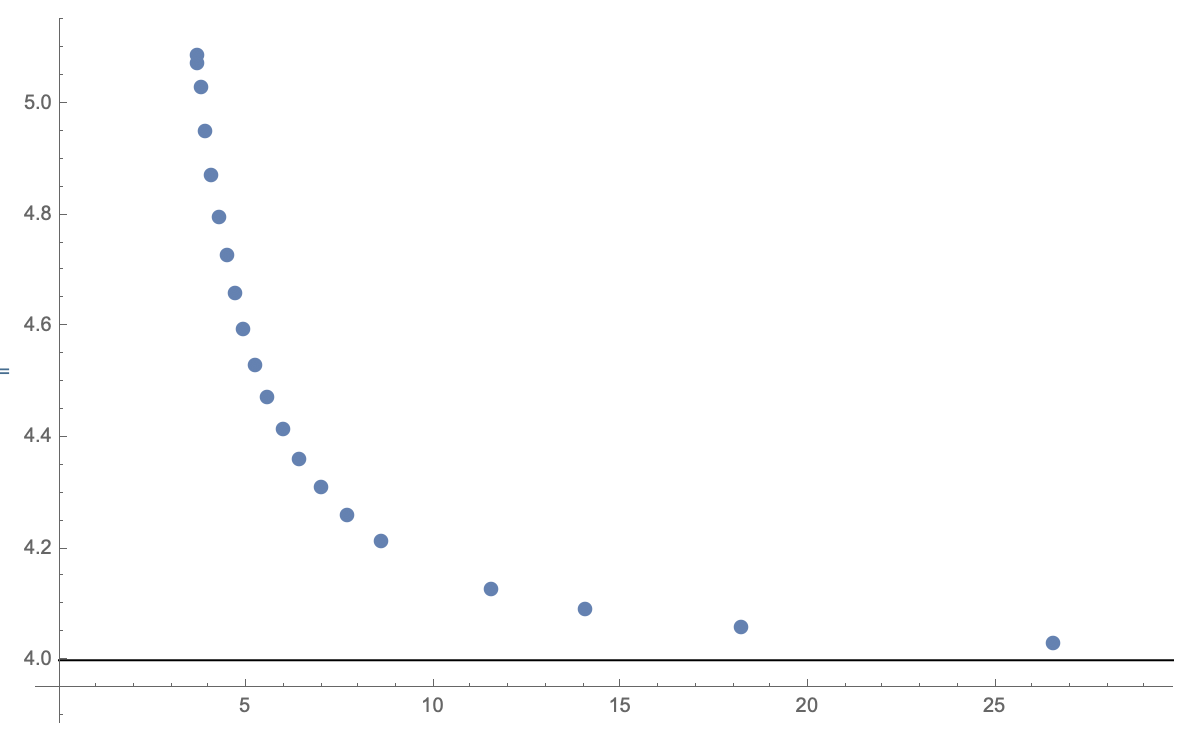}
\caption{Here, for $G_{1/2}$, we have plotted points on $\partial_0 N_+$, as $x_0$ varies from $0.6$ to $0.98$ in increments of $0.02$.}
\end{figure}
\begin{figure}[H]
\centering
\includegraphics[width=\textwidth]{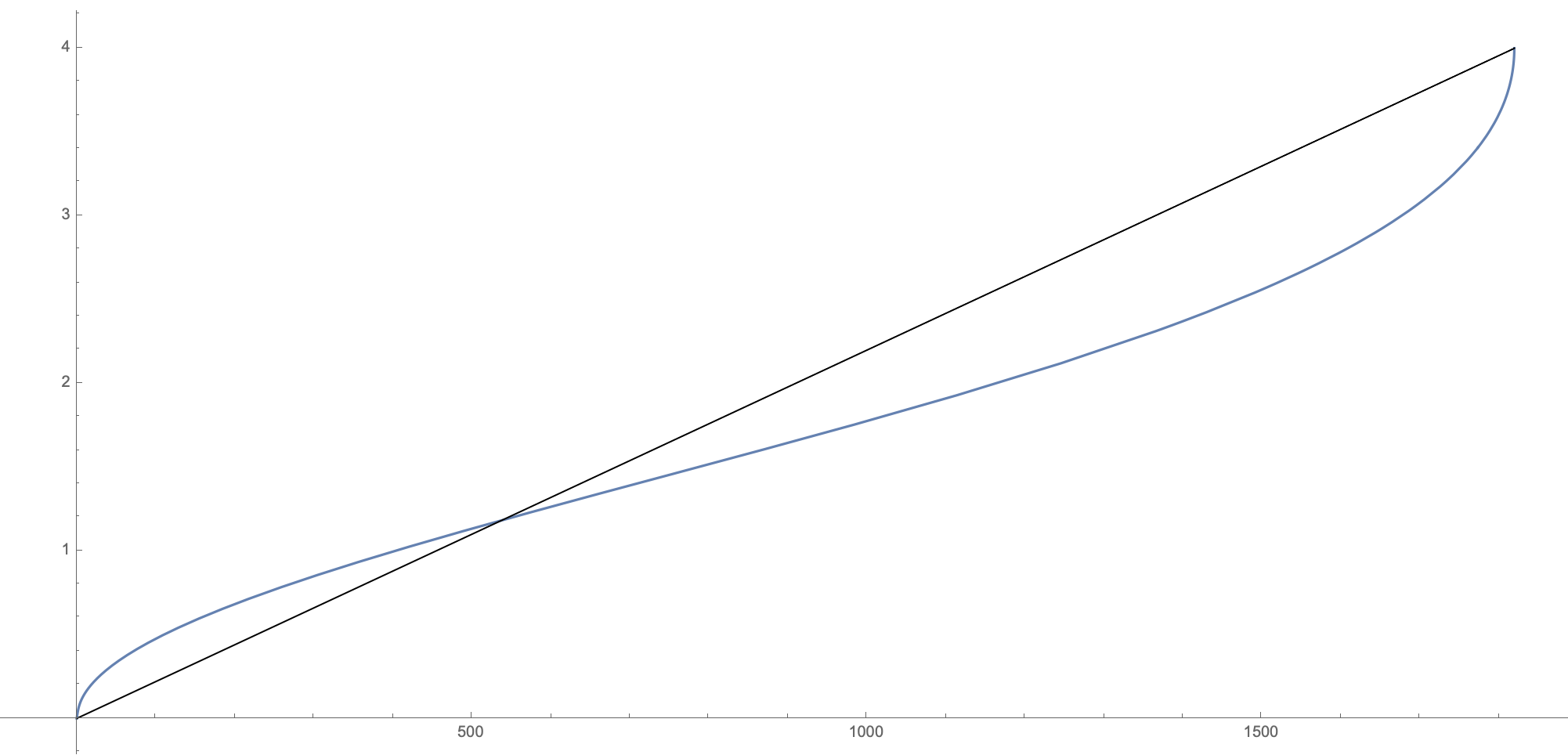}
\caption{This depicts a the image of $\Lambda_P$ for $\alpha=1/2$ and $x_0=0.99945$ over the interval $(0,\rho]$.}
\end{figure} 
\section{Proof of the Bounding Box Theorem}
We now study the system of ODE's that governs the behavior of $x,y,z,a,$ and $b$ as in equations $(6)$ and $(8)$.
We write 
$\lambda_{t+\epsilon}=u|\lambda_t|v$,
where $u$ is the flowline
connecting $p_{t+\epsilon}$ to $p_t$ and $v$ is the flowline 
connecting $\hat{p}_t$ to $\hat p_{t+\epsilon}$.
We have
$$
(a',b',0)=\Lambda'_P(t)=\lim_{\epsilon \to 0} \frac{\Lambda_P(t+\epsilon)-\Lambda(t)}{\epsilon},
$$
$$
  \Lambda_P(t+\epsilon) \approx (\epsilon x, \epsilon y,\epsilon z) *
  (a,b,0) * ( \epsilon x, \epsilon y,-\epsilon z).
$$
The approximation is true up to order $\epsilon^2$ and
$(*)$ denotes multiplication in $G_\alpha$.
A direct calculation gives
\begin{equation} a'=2x+az \textrm{ and  } b'=2y-\alpha bz.\end{equation}

Simply from its differential equation, it is evident that $a'>0$ on $(0,\rho)$. This implies that
$\Lambda_P(t)$ is the graph of a function for $t\in (0,\rho)$, hence $\Lambda_P(t)$
avoids the vertical sides of $B_P$. This is the first, easy step in proving the Bounding Box Theorem.

To finish the proof, it would suffice to show that $\Lambda_P(t)$ also avoids the horizontal sides of $B_P$, which amounts to proving that $b'(t)>0$ for all $t\in(0,\rho]$. A priori, it is not evident that $b'>0$ in this interval. For example, the function $b$ may start out concave as depicted in Figure 7. Also, after the half-period $\rho$, $b'$ may actually be negative as depicted in Figure 8. However, the remarkable fact that $b'>0$ in $(0,\rho]$ for all choices of $\alpha$ and $x_0$ is also true, and we demonstrate this fact in what follows.
\begin{figure}[H]
\centering
\includegraphics[width=0.8\textwidth]{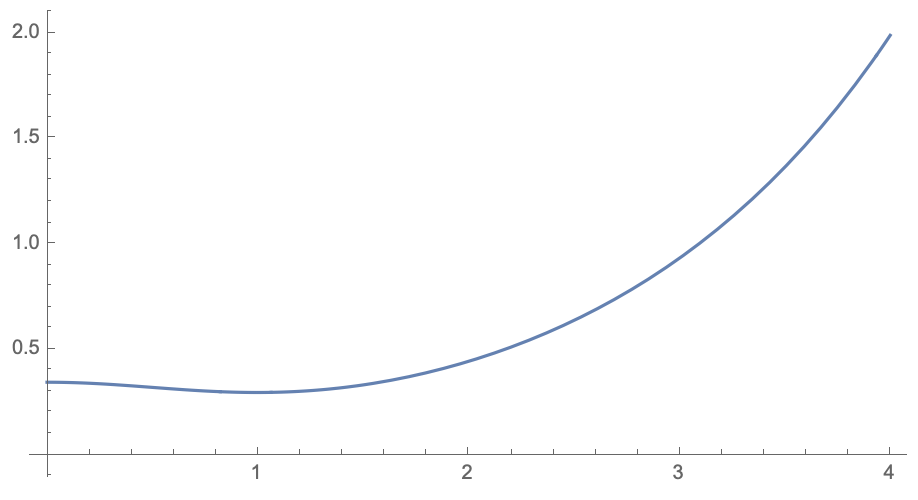}
\caption{This is the graph of $b'(t)$ over the interval $(0,\rho)$ for the parameter choices $\alpha=3/4$ and $x(0)=.985$. We can see that $b'$ is initially decreasing.}
\end{figure}
\begin{figure}[H]
\centering
\includegraphics[width=0.57\textwidth]{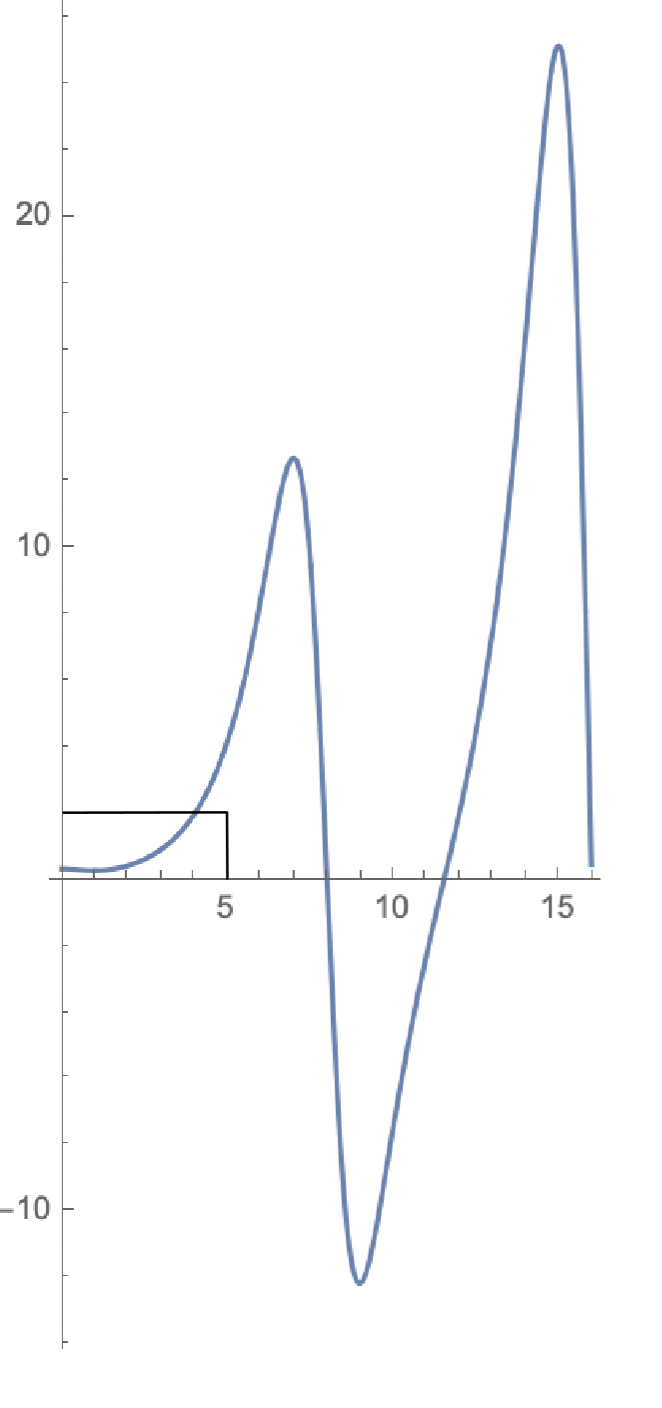}
\caption{This is the graph of $b'(t)$ over the interval $(0,4\rho)$ for the parameter choices $\alpha=3/4$ and $x(0)=.985$, illustrating that $b'$ is not necessarily positive outside of the half-period interval. Figure 7 is contained in the small box shown in this figure.}
\end{figure}
Once again, we collect the ODE's of interest to us together:
$$x'=-xz\quad y'=\alpha y z\quad z'=x^2-\alpha y^2\quad b'=2y-\alpha b z$$ from which we compute
\begin{equation}z''=-2z(x^2+\alpha y^2)\quad b''=\alpha b(\alpha z^2-z').\end{equation}
\begin{lemma}
To show that $b'>0$ in $(0,\rho),$ it suffices to show that $b'>0$ whenever $b''<0$ in the interval $(0,\rho).$ 
\end{lemma}
\begin{proof}
We remark that the function $b''$ changes sign at the unique point $t_0,$ where $\alpha z(t_0)^2=z'(t_0)$. Since $x,y,z>0$ in the whole interval $(0,\rho),$ and since $x(0)>y(0),$ we get that $z'(0)>0.$ Thus, $b''(0)<0.$ At the point $t_0$ where $b$ changes concavity, we have $x(t_0)=\sqrt{\frac{\alpha}{\alpha+1}}.$ Since $x$ is monotonically decreasing in our interval $(0,\rho),$ it follows that $b''$ changes sign at most once in the interval $(0,\rho).$ By elementary calculus, it follows that to show $b'>0$ in $(0,\rho),$ it suffices to show $b'>0$ in the interval $(0,t_0)$ if $t_0<\rho,$ or else on the whole interval $(0,\rho).$ Either case amounts to showing that $b'>0$ whenever $b''<0$ in the interval $(0,\rho)$.
\end{proof}
We observe that $z''<0$ everywhere in $(0,\rho)$. Also, whenever $b''<0,$ we have that $z'>\alpha z^2>0$. We first get an inequality regarding the function $z(t):$ 
\begin{lemma}
$$2\alpha \int_0^t z(s) ds\geq t\alpha z(t)$$
\end{lemma}
\begin{proof}
By the well known Hermite-Hadamard Inequality for concave functions, the concavity of $z$ in $(0,\rho),$ and the fact that $z(0)=0,$ we get:
$$\frac{1}{t}\int_0^t z(s) ds\geq \frac{z(t)+z(0)}{2}=z(t)/2.$$
Multiplying by $2\alpha t>0,$ we get our desired inequality.
\end{proof}
We now recall the Log-Convex Version of Hermite-Hadamard, proven first in \cite{GPP}:
\begin{proposition}[\cite{GPP}]
If $f$ is log-convex on $[a,b]$ then
$$\frac{1}{b-a}\int_a^b f(s) ds\leq \frac{f(b)-f(a)}{\log f(b)-\log f(a)}.$$
\end{proposition}
Let's return to one of our initial ODE's: $y'=\alpha yz.$ Dividing by $y,$ integrating, and multiplying by 2, we get:
\begin{equation}2\log y(t) -2\log y(0)= 2\alpha \int_0^t z(s) ds. \end{equation}
Since $z'>0$ whenever $b''<0,$ we get that $y^2$ is log-convex whenever $b''<0.$ We can now get:
\begin{lemma} For all $t$ where $b''(t)<0,$ we have
$$\int_0^t y(s)^2 ds \leq t\cdot\frac{y(t)^2}{2\log y(t)/y(0)}$$
\end{lemma}
\begin{proof}
By Proposition 4.3, we have
$$\int_0^t y(s)^2 ds \leq t\cdot\frac{y(t)^2-y(0)^2}{2\log y(t)/y(0)}\leq t\cdot\frac{y(t)^2}{2\log y(t)/y(0)}$$
where the last inequality comes from $y(0)>0.$ 
\end{proof}
We are now in a position to prove the Bounding Box Theorem:
\begin{proof}
By Lemma 4.1, it suffices to show that $b'>0$ whenever $b''<0.$ If we integrate the ODE for $b$, we get
\begin{equation} b(t)=\frac{2}{y(t)}\int_0^t y(s)^2 ds.\end{equation}
Differentiating, we want to show $y(t)^3-y'(t)\int_0^t y(s)^2 ds\geq 0$ whenever $b''<0.$ Equivalently (we can divide by $y,y'$ since they are always strictly greater than $0$):
$$ \int_0^t y(s)^2 ds\leq \frac{y(t)^3}{y'(t)}.$$
By Lemma 4.4 it suffices to show
$$t\cdot\frac{y(t)^2}{2\log y(t)/y(0)}\leq\frac{y(t)^3}{y'(t)}$$
or, by cancelling some terms and taking the reciprocal,
$$2\log y(t)/y(0) \geq t\cdot \frac{y'(t)}{y(t)}.$$
Since $y'=\alpha yz$, this is equivalent to
$$2\alpha \int_0^t z(s) ds\geq t\alpha z(t)$$
which is nothing but the inequality of Lemma 4.2.
\end{proof}
Finally, since $b'>0, \Lambda_P(t)$ is an increasing function, so $\Lambda_P(t)$ avoids the vertical sides of $B_P$ when $t\in(0,\rho)$. This, along with the previously stated fact that $\Lambda_P(t)$ avoids the horizontal sides of $B_P$, finishes the proof of the Bounding Box Theorem. 
\section{Proof of the Monotonicity Theorem}
\subsection{Endpoints of Symmetric Flowlines}
Henceforth, we view $\partial_0 N_+$ as the following parametrized curve in the $XY$ plane. Denote $x(0)=x_0$, then
$$\partial_0 N_+= \{ (a_{x_0}(P(x_0)/2),b_{x_0}(P(x_0)/2),0)\}, \textrm{ as } x_0 \textrm{ varies in } \bigg(\sqrt{\frac{\alpha}{1+\alpha}},1\bigg).$$
To prove Equation $10$ in general it would suffice, by using the Bounding Box Theorem, to prove $B_P\cap \partial_+ N=\emptyset$, and, to prove the latter statement, it suffices to show that $\partial_0 N_+$ is the graph of a decreasing function in Cartesian coordinates. This involves differentiating our ODE's with respect to the initial value $x_0$. Let $\bar{x}$ denote $dx(t,x_0)/dx_0$, etc. Then we get
$$\bar{x}'=-x\bar{z}-z\bar{x},\quad \bar{y}'=\alpha y\bar{z}+\alpha z\bar{y},\quad \bar{z}'=2x\bar{x}-2\alpha y \bar{y},$$
$$\bar{a}'=2\bar{x}+a\bar{x}+x\bar{a}, \quad \bar{b}'=2\bar{y}-\alpha\bar{y}b-\alpha y\bar{b}.$$
Since $x^2+y^2+z^2=1$ for all $t, x_0,$ it follows that 
\begin{equation}x\bar{x}+y\bar{y}+z\bar{z}=0\end{equation}
always, and we can get a similar equation for the time derivative. Now, we prove some very useful propositions:
\begin{proposition}
$ax-\alpha by = 2z$ for all $t$ and $x_0$.
\end{proposition}
\begin{proof}
By integrating the ODE's in time for $a$ and $b$ we get:
$$ax=2\int_0^t x(s)^2 ds \textrm{ and } \alpha by=2\int_0^t \alpha y(s)^2 ds$$
hence 
$$ax-\alpha by= 2\int_0^t x(s)^2-\alpha y(s)^2 ds= 2z(t)-2z(0)=2z$$
since $z(0)=0$ always. Note: this gives us our promised second proof of the Reciprocity Lemma, if we evaluate at $t=\rho$.
\end{proof}
Since the above equality is true for all $t$ and $x_0$ we can differentiate with respect to $x_0$ and get
\begin{corollary}
$a\bar{x}+x\bar{a}-\alpha b\bar{y}-\alpha y \bar{b}=2\bar{z}$ for all $t$ and $x_0$.
\end{corollary}
Now we prove:
\begin{proposition}
$x\bar{a}+y\bar{b}=0$ for all $t$ and $x_0$.
\end{proposition}
\begin{proof}
Integrating the ODE's (in time) for $\bar{a}$ and $\bar{b}$ gets us:
$$x\bar{a}=\int 2x\bar{x}+ax\bar{z} \textrm{ and } y\bar{b}=\int 2y\bar{y}-\alpha by\bar{z}.$$
Adding the two integrals above and using equation $(16)$ gets us our desired equality.
\end{proof}
Corollary 5.2 and Proposition 5.3 combine to get us the following useful expressions for $\bar{a}$ and $\bar{b}$. 
\begin{corollary}
We have
$$\bar{a}=\frac{1}{x(1+\alpha)}\bigg(2\bar{z}+\alpha b\bar{y}-a\bar{x}\bigg)$$
and
$$\bar{b}=-\frac{1}{y(1+\alpha)}\bigg(2\bar{z}+\alpha b\bar{y}-a\bar{x}\bigg).$$
\end{corollary}
\subsection{The Monotonicity Theorem for $G_{1/2}$}
To get our desired results about $\partial_0 N_+$, we must look at a particular case of our one-parameter family, where we have more information about the derivative of $P(x_0)$, courtesy of the expression of $P$ in terms of an elliptic integral. Everything we have heretofore shown is, however, applicable to every $G_\alpha$ with $0<\alpha\leq 1$.
Although we restrict ourselves to $G_{1/2}$, the methods presented here could just as well be applied to $G_1$, where we also have an explicit formula for the period function. 

In this section, we show that $\partial_0 N_+$ is the graph of a non-increasing function in Cartesian coordinates (for $G_{1/2}$) by using properties of the period function. This finishes the proof of Equation \ref{goal}, which in turn allows us to prove our main theorem. For encouragement, we refer the reader back to Figure 5, where we can see that $\partial_0 N_+$ is indeed the graph of a non-increasing function in Cartesian coordinates. It will be relatively easy to show that $\partial_0N_+$ is the graph of a Cartesian function, but the proof that $\partial_0N_+$ is a non-increasing function will be more involved. For example, we will first need to show that $\partial_0N_+$ limits to the line $b=4$ as $x_0\to 1$. 

 Recall Proposition 3.5, which states that $P(\beta)$ is decreasing with respect to $\beta$ in the case when $\alpha=1$ or $1/2$. Here, we change variables for the period function from $\beta$ to $x_0\in (1/\sqrt{3},1)$ and get:
\begin{proposition}
$$\frac{d}{dx_0}\big(P(x_0)\big) >0$$
\end{proposition}
\begin{proof}
This amounts to an application of the chain rule. Since the vector (associated to $\beta$) as in the statement of Theorem 3.1 is on the same flow line as the vector associated to $x_0$, we know
$$x_0^{1/2}\sqrt{1-x_0^2}=\bigg(\beta\sqrt{\frac{1}{3}}\bigg)^{1/2}\frac{\beta\sqrt{2}}{\sqrt{3}}$$
so
$$\beta^3=\frac{3\sqrt{3}}{2}(x_0-x_0^3),$$
thus
$$\frac{d\beta}{dx_0}=\frac{\sqrt{3}}{2\beta^2}(1-3x_0^2)<0$$
since $x_0^2>1/3$.
By the chain rule and Proposition 3.5, we get our desired result.
\end{proof}
Since $z$ always vanishes at the half period, we have
\begin{proposition}
For any initial value $x_0$, we have
$$\bar{z}+(\frac{1}{2}\frac{dP}{dx_0})z'=0$$
at the time $t=P(x_0)/2$. Also, by Proposition 5.5, and since $z'<0$ at the half period, we get that 
$$\bar{z}(P(x_0)/2)>0, \quad \forall x_0.$$
\end{proposition}
From equation $(9)$ and the fact that $x'$ and $y'$ always vanish at the half-period, we have
\begin{proposition}
$$\bar{y}(P(x_0)/2)=\frac{1}{2\sqrt{x_0}}>0 \quad \textrm{ and }  \quad \bar{x}(P(x_0)/2)=-2x_0<0$$
\end{proposition}
We are ready to get some information about $\partial_0 N_+$, beginning with:
\begin{corollary} 
$\partial_0 N_+$ is the graph of a function in Cartesian coordinates.
\end{corollary}
\begin{proof}
This is equivalent to showing that 
$$\frac{d}{dx_0}a_{x_0}(P(x_0)/2)>0.$$
The chain rule gets us
$$\frac{d}{dx_0}a_{x_0}(P(x_0)/2)=\bar{a}(P(x_0)/2)+(\frac{1}{2}\frac{dP(x_0)}{dx_0})a'(P(x_0)/2)$$
$$=\bigg(\frac{2}{3x}\bigg(2\bar{z}+\frac{1}{2}b\bar{y}-a\bar{x}\bigg)+(x+az/2)\frac{dP}{dx_0}\bigg)\bigg\rvert_{t=P(x_0)/2}$$
By Propositions 5.5, 5.6, and 5.7, we know that all the terms above are positive at $P(x_0)/2$, whence the desired result.
\end{proof}
As with showing that $b'>0$ in the interval $(0,\rho)$, things are more difficult with the function $b$. We need three lemmas first. The proof of the following may also suggest that $\alpha=1/2$ is a "special case"; nevertheless, the situation is different than for Sol. Richard Schwartz proves a similar limit in \cite{S} for Sol (in which case, the limit is $2$), but his method uses an additional symmetry of the flow lines that we cannot use here. 
\begin{lemma}For $\alpha=1/2$,
$$\lim_{x_0\to 1} b_{x_0}(P(x_0)/2)=4$$
\end{lemma}
\begin{proof}
Recall equation (15), which gives us an integral form for $b$:
$$b_{x_0}(P(x_0)/2)=\frac{2}{y(P(x_0)/2)}\int_0^{P(x_0)/2} y(s)^2 ds.$$
From Equation (9), we know that $\lim_{x_0\to 1} y(P(x_0)/2)=1$, so it suffices to show 
$$\lim_{x_0\to1}\int_0^{P(x_0)/2} y(s)^2 ds=2.$$
Now, a minor miracle occurs. There is a second order ODE for $y$
$$2yy''=y^2-\frac{3}{2}y^4-2(y')^2$$
from which we can also get an ODE for $y^2$:
$$(y^2)''=y^2-(3/2)y^4.$$
Here we remark that the cases $\alpha=1$ and $1/2$ are the only ones where we do not also get an additional, unpleasant, and nonlinear term: $(y')^2$. 
Let $f=y^2$, then we can solve the ODE $f''=f-(3/2)f^2$ for $f$ and get (as done, for example, in \cite{CV}):
$$f(t,x_0)=\nu_1+\nu_2\textrm{dn}(t\nu_3,\nu_4)^2$$
where $\textrm{dn}(u,m)$ is the Jacobi elliptic function (with parameters as in Mathematica) and with
$$\nu_1=\frac{1}{2}(x_0^2 - \sqrt{4 x_0^2 - 3 x_0^4})$$
$$\nu_2=\frac{1}{2}(2 - 3 x_0^2 + \sqrt{4 x_0^2 - 3 x_0^4})$$
$$\nu_3=\frac{\sqrt{2 - 3 x_0^2 + \sqrt{4 x_0^2 - 3 x_0^4}}}{2\sqrt{2}}$$
$$\nu_4=\frac{2-3x_0^2-\sqrt{4x_0^2-3x_0^4}}{2-3x_0^2+\sqrt{4x_0^2-3x_0^4}}.$$
Since the function is periodic with period $P(x_0)$, and since we know that $\textrm{dn}(u,k)^2$ has period $2K(k)$, where $K$ is the complete elliptic integral of the first kind, we get the identity
\begin{equation}P(x_0)=\frac{2K(\nu_4)}{\nu_3}.\end{equation}
Now, in \cite{T}, many properties of elliptic integrals and elliptic functions are reviewed, including:
$$E(k)=\int_0^{K(k)} \textrm{dn}(u,k)^2du$$
where $E$ is the complete elliptic integral of the second kind. Therefore, we can write our original integral as
$$\int_0^{P(x_0)/2} y(s)^2 ds=\frac{\nu_1}{2}P(x_0)+\frac{\nu_2}{\nu_3}E(\nu_4).$$
As $x_0$ tends to $1$, the term $\nu_1P(x_0)/2$ converges to $0$ and the term $\frac{\nu_2}{\nu_3}E(\nu_4)$ converges to 2, which is what we desired to prove.
\end{proof}
More generally, we have the following conjecture
\begin{conjecture} Let
$$L(\alpha):=\lim_{x_0\to 1} b_{x_0,\alpha}(P(x_0,\alpha)/2), \textrm{ defined for all } \alpha \in (0,1].$$
We conjecture that $L(\alpha)$ is monotonically decreasing from $\alpha=0$ to $\alpha=1$ and that $\lim_{\alpha\to 0} L(\alpha)=\infty$. In fact, we also conjecture that 
$$L(\alpha)=\frac{2}{\alpha}.$$
\end{conjecture}
The next technical lemma is quite wearisome; we have placed its demonstration in Appendix A. We direct the reader who does not wish to plough through its proof to Figure 9, which provides numerical evidence for its veracity.
\begin{lemma}
For $\alpha=1/2,$ we have
$$G(x_0):=\frac{dP}{dx_0}(x_0)-\pi\bigg(\frac{1}{2\sqrt{x_0}}+\frac{2x_0\sqrt{x_0}}{1-x_0^2}\bigg)<0,\quad \forall x_0 \in \bigg(\frac{1}{\sqrt{3}},1\bigg)$$ 
\end{lemma}
\begin{figure}[H]
\centering
\includegraphics[width=\textwidth]{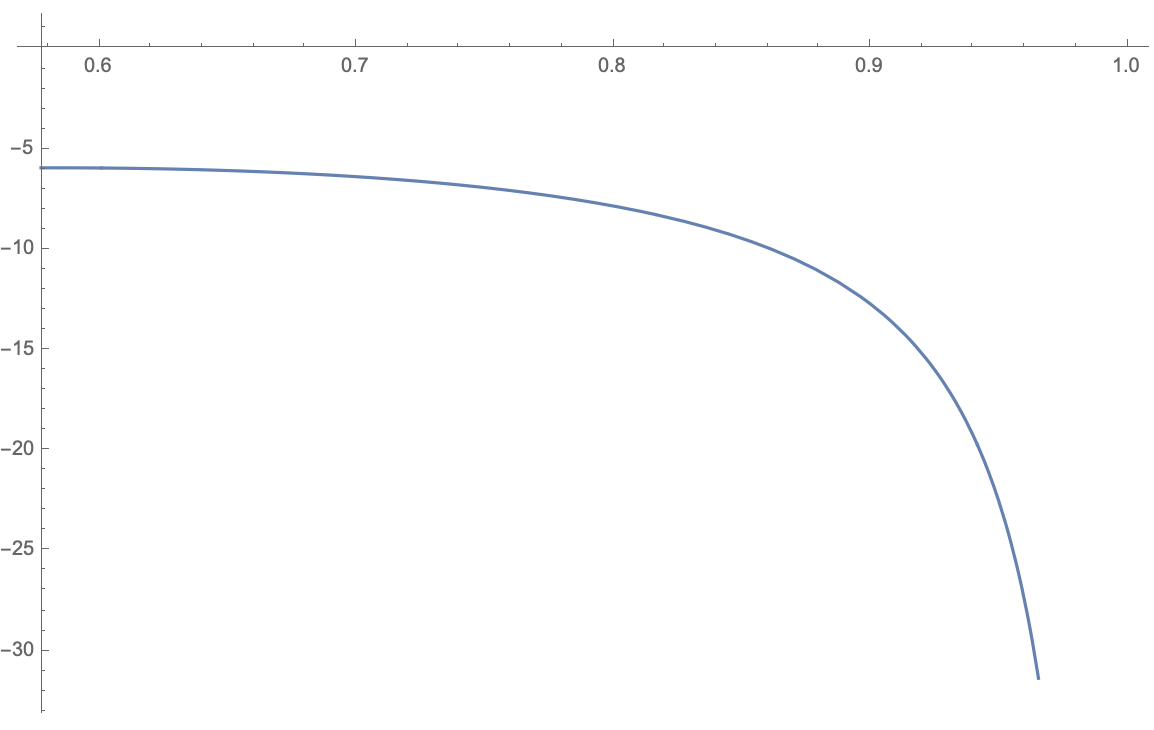}
\caption{The graph of the function $G(x_0)$.}
\end{figure}
The penultimate step towards the Monotonicity Theorem:
\begin{lemma}
$$\frac{d}{dx_0}b_{x_0}(P(x_0)/2)<0, \textrm{ whenever } b_{x_0}(P(x_0)/2)>\pi$$
\end{lemma}
\begin{proof}
By the chain rule we know (as we did for Corollary 5.4)
$$\frac{d}{dx_0}b_{x_0}(P(x_0)/2)=\bigg(-\frac{2}{3y}\bigg(2\bar{z}+\frac{1}{2}b\bar{y}-a\bar{x}\bigg)+(y-\frac{1}{4}bz)\frac{dP}{dx_0}\bigg)\bigg\rvert_{t=P(x_0)/2}.$$
The $bz$ term is zero at the half period, and we can simplify further. The above is less than zero if and only if
$$(2x^2+2y^2)\frac{dP}{dx_0}<b\bar{y}-2a\bar{x} \textrm{ at } t=P(x_0)/2.$$
However we know $x^2+y^2=1$ at the half period, and we may also employ the Reciprocity Lemma (equivalently, Theorem 3.15 or Proposition 5.1). This yields that our desired result is equivalent to showing
$$\frac{dP}{dx_0}(x_0)<\big(b(\bar{y}-\frac{y}{x}\bar{x})\big)\big\rvert_{t=P(x_0)/2}.$$
By hypothesis, $b_{x_0}(P(x_0)/2)>\pi$, and we can use equation $(9)$ and Proposition 5.7. This means it suffices to show
$$\frac{dP}{dx_0}(x_0)<\pi\bigg(\frac{1}{2\sqrt{x_0}}+\frac{2x_0\sqrt{x_0}}{1-x_0^2}\bigg),\quad \forall x_0 \in \bigg(\frac{1}{\sqrt{3}},1\bigg)$$
which is nothing but Lemma 5.11.
\end{proof}
Now we are ready to prove the Monotonicity Theorem.
\begin{theorem}[The Monotonicity Theorem]
For $\alpha=1/2$, $\partial_0 N_+$ is the graph of a non-increasing function (in Cartesian coordinates).
\end{theorem}
\begin{proof}
This is equivalent to showing 
$$\frac{d}{dx_0}b_{x_0}(P(x_0)/2)\leq0$$
or, as in the proof of the previous lemma,
$$\frac{dP}{dx_0}(x_0)\leq\big(b(\bar{y}-\frac{y}{x}\bar{x})\big)\big\rvert_{t=P(x_0)/2}.$$
We now show that $b_{x_0}(P(x_0)/2)\geq4$ always. If this is true, then the theorem follows by Lemma 5.12. Otherwise, assume that  $b_{x_0}(P(x_0)/2)<4$ for some choice of $x_0$. By Lemma 5.9, $b_{x_0}(P(x_0)/2)$ limits to $4$ as $x_0$ tends to $1$, so $b_{x_0}(P(x_0)/2)$ must eventually be greater than $\pi$. Let $x_0'$ be the last initial value where $b(P/2)\leq \pi$. By Lemma 5.12, $b_{x_0}(P(x_0)/2)$ will be strictly decreasing in $x_0$, which is a contradiction. Hence, $b>4$ always and the theorem is proven. 
\end{proof}
\section{Appendix A: Proof of Lemma 5.11}
\begin{lemma*}[5.11]
For $\alpha=1/2,$ we have
$$G(x_0):=\frac{dP}{dx_0}(x_0)-\pi\bigg(\frac{1}{2\sqrt{x_0}}+\frac{2x_0\sqrt{x_0}}{1-x_0^2}\bigg)<0,\quad \forall x_0 \in \bigg(\frac{1}{\sqrt{3}},1\bigg)$$ 
\end{lemma*}
\begin{proof}
In the course of proving Lemma 5.9, we defined the functions
$$\nu_3(x_0)=\frac{\sqrt{2 - 3 x_0^2 + \sqrt{4 x_0^2 - 3 x_0^4}}}{2\sqrt{2}}$$
and
$$\nu_4(x_0)=\frac{2-3x_0^2-\sqrt{4x_0^2-3x_0^4}}{2-3x_0^2+\sqrt{4x_0^2-3x_0^4}}.$$
To avoid a disarray of variables, we will omit writing the dependence of $\nu_3$ and $\nu_4$ on $x_0$ explicitly. Recall Equation 17, where we found the following expression for the period function in $G_{1/2}$:
$$P(x_0)=\frac{2K(\nu_4)}{\nu_3}$$
where the parameter for the complete elliptic integral $K$ is as in Mathematica. We use the imaginary-modulus transformation in \cite[Eq.~19.7.5]{DLMF} to change the argument of $K$ to something more appropriate:
$$P(x_0)=\frac{2K(\nu_4)}{\nu_3}=\frac{2}{\nu_3 \sqrt{1-\nu_4}}\cdot K\bigg(\frac{\nu_4}{\nu_4-1}\bigg)$$
We prefer the above expression since the argument for $K$ is now a diffeomorphism from $[1/\sqrt{3},1)$ onto $[0,1)$. Again, for brevity's sake, we define the following rational functions:
$$\sigma_1(x_0)=\frac{2}{\nu_3\sqrt{1-\nu_4}}$$
$$\sigma_2(x_0)=\frac{\nu_4}{\nu_4-1}$$
to get the following expression for the period function
$$P(x_0)=\sigma_1(x_0)K(\sigma_2(x_0)).$$
Remembering the differentiation formula for $K$ (also given in \cite[Eq. ~19.4.1]{DLMF}), we get (again omitting dependences on $x_0$)
$$\frac{dP}{dx_0}=K(\sigma_2)\bigg(\frac{d\sigma_1}{dx_0}-\frac{\sigma_1}{2\sigma_2}\cdot\frac{d\sigma_2}{dx_0}\bigg)+E(\sigma_2)\cdot\frac{\sigma_1\cdot\frac{d\sigma_2}{dx_0}}{2\sigma_2(1-\sigma_2)}$$
Thus, proving this lemma amounts to demonstrating that
\begin{equation}K(\sigma_2)\bigg(\frac{d\sigma_1}{dx_0}-\frac{\sigma_1}{2\sigma_2}\cdot\frac{d\sigma_2}{dx_0}\bigg)+E(\sigma_2)\cdot\frac{\sigma_1\cdot\frac{d\sigma_2}{dx_0}}{2\sigma_2(1-\sigma_2)}-\pi \bigg(\frac{1}{2\sqrt{x_0}}+\frac{2x_0\sqrt{x_0}}{1-x_0^2}\bigg)<0\end{equation}
We will use the following inequality, proven in \cite{AQ}, that also gives a good approximation of $K$,
\begin{equation}
\frac{\pi}{2}\bigg(\frac{\arctanh{r}}{r}\bigg)^{1/2}<K(r^2)<\frac{\pi}{2}\bigg(\frac{\arctanh{r}}{r}\bigg), \quad \forall r, 0<r<1
\end{equation}
Before continuing, we wish to emphasize that the parameter for $K$ above is as in Mathematica for the \textbf{EllipticK} function, which causes a difference (albeit superficial) between our statement of the inequality and the statement given in \cite{AQ}. Now, it can be verified with Mathematica that the quantity 
$$\bigg(\frac{d\sigma_1}{dx_0}-\frac{\sigma_1}{2\sigma_2}\cdot\frac{d\sigma_2}{dx_0}\bigg)$$ is always negative and
$$\frac{\sigma_1\cdot\frac{d\sigma_2}{dx_0}}{2\sigma_2(1-\sigma_2)}$$ is always positive for $x_0 \in (1/\sqrt{3},1)$.
Thus, using the inequality given in $(19)$ (and the well-known fact that $E(k)\leq \pi/2$ for any real $0<k<1$), we know that
$$K(\sigma_2)\bigg(\frac{d\sigma_1}{dx_0}-\frac{\sigma_1}{2\sigma_2}\cdot\frac{d\sigma_2}{dx_0}\bigg)+E(\sigma_2)\cdot\frac{\sigma_1\cdot\frac{d\sigma_2}{dx_0}}{2\sigma_2(1-\sigma_2)}-\pi \bigg(\frac{1}{2\sqrt{x_0}}+\frac{2x_0\sqrt{x_0}}{1-x_0^2}\bigg)<$$
$$\frac{\pi}{2}\bigg(\frac{\arctanh{\sqrt{\sigma_2}}}{\sqrt{\sigma_2}}\bigg)^{1/2}\bigg(\frac{d\sigma_1}{dx_0}-\frac{\sigma_1}{2\sigma_2}\cdot\frac{d\sigma_2}{dx_0}\bigg)+\frac{\pi}{2}\cdot \frac{\sigma_1\cdot\frac{d\sigma_2}{dx_0}}{2\sigma_2(1-\sigma_2)}-\pi \bigg(\frac{1}{2\sqrt{x_0}}+\frac{2x_0\sqrt{x_0}}{1-x_0^2}\bigg)$$
It follows that, to finish our proof, it suffices to prove
$$\bigg(\frac{\arctanh{\sqrt{\sigma_2}}}{\sqrt{\sigma_2}}\bigg)^{1/2}\bigg(\frac{d\sigma_1}{dx_0}-\frac{\sigma_1}{2\sigma_2}\cdot\frac{d\sigma_2}{dx_0}\bigg)+\frac{\sigma_1\cdot\frac{d\sigma_2}{dx_0}}{2\sigma_2(1-\sigma_2)}-2\bigg(\frac{1}{2\sqrt{x_0}}+\frac{2x_0\sqrt{x_0}}{1-x_0^2}\bigg)<0.$$
We remember that $$\bigg(\frac{d\sigma_1}{dx_0}-\frac{\sigma_1}{2\sigma_2}\cdot\frac{d\sigma_2}{dx_0}\bigg)$$ is always negative, so it in fact suffices to show:
$$\frac{\arctanh{\sqrt{\sigma_2}}}{\sqrt{\sigma_2}}>\bigg(2\bigg(\frac{1}{2\sqrt{x_0}}+\frac{2x_0\sqrt{x_0}}{1-x_0^2}\bigg)-\frac{\sigma_1\cdot\frac{d\sigma_2}{dx_0}}{2\sigma_2(1-\sigma_2)}\bigg)^{2}\cdot\frac{1}{\bigg(\frac{d\sigma_1}{dx_0}-\frac{\sigma_1}{2\sigma_2}\cdot\frac{d\sigma_2}{dx_0}\bigg)^{2}}.$$
Using elementary calculus, it can be proven that the function
$$F(x)=\frac{\arctanh{x}}{x}$$
satisfies 
$$F(x)\geq1$$
with equality if and only if $x=0$. We can simplify $\sigma_2$ and see that
$$\sigma_2(x_0)= \frac{-2+3x_0^2+\sqrt{4x_0^2-3x_0^4}}{2\sqrt{4x_0^2-3x_0^4}}$$
so $\sigma_2(x_0)=0$ if and only if $x_0=\pm 1/\sqrt{3}$. It follows that
$$\frac{\arctanh{\sqrt{\sigma_2}}}{\sqrt{\sigma_2}}\geq 1$$
for $x_0\in[1/\sqrt{3},1)$ and the inequality is strict unless $x_0=1/\sqrt{3}$. So, we have reduced our desired inequality $(18)$ to showing that
$$1>\bigg(2\bigg(\frac{1}{2\sqrt{x_0}}+\frac{2x_0\sqrt{x_0}}{1-x_0^2}\bigg)-\frac{\sigma_1\cdot\frac{d\sigma_2}{dx_0}}{2\sigma_2(1-\sigma_2)}\bigg)^{2}\cdot\frac{1}{\bigg(\frac{d\sigma_1}{dx_0}-\frac{\sigma_1}{2\sigma_2}\cdot\frac{d\sigma_2}{dx_0}\bigg)^{2}}$$
for $x_0\in (1/\sqrt{3},1)$. After simplification, the desired inequality is
\tiny
\begin{equation}\frac{\left(27 x_0^6-36 x_0^4-3 x_0^2+8 \sqrt[4]{4-3
   x_0^2}+4\right)^2 \left(-3 x_0^2+\sqrt{4-3 x_0^2}
   x_0+2\right)^4}{64 \sqrt{4-3 x_0^2} \left(-27 x_0^8+72
   x_0^6-57 x_0^4+12 x_0^2+6 \sqrt{4-3 x_0^2} x_0-9
   \sqrt{4-3 x_0^2} x_0^7+18 \sqrt{4-3 x_0^2} x_0^5-17 \sqrt{4-3
   x_0^2} x_0^3+2\right)^2}<1\end{equation}\normalsize
for $x_0\in(1/\sqrt{3},1)$. Fortunately for us, Mathematica (and presumably other computer algebra systems) can verify inequality $(20)$ very quickly, which allows us to conclude our proof. 

We have included the Mathematica code that proves this last inequality in Appendix B.
\end{proof}
\section{Appendix B: Computer Code}
Here we present the Mathematica code that generates the figures we have presented. In addition, we present the modification to Richard Schwartz's Java program, \cite{S2}, which allows the rendering of geodesics and geodesic spheres in $G_\alpha$ groups (Figure 3), the Mathematica code written by Stephen Miller that allows numerical computation of the period function for arbitrary positive $\alpha$, and the Mathematica code that verifies the inequality in $(20)$. We note that the whole of the program \cite{S2} is available online, on Richard Schwartz's website.
\subsection{Figures 1 and 2}
This gets us the implicit plots in the plane:
\begin{lstlisting}[language=Mathematica]
Manipulate[
 ContourPlot[(1/a) E^(2*a*z) + E^(-2 z) + 
    w^2 == (1+a)/(a*b^2), {w,-10,10}, {z,-10,10}, 
  PlotRange -> All, AspectRatio -> Automatic, PlotPoints -> 50], 
  {b,0.001,1}, {a,0.001,1}]
\end{lstlisting}
A 3-D rendering can be obtained with:
\begin{lstlisting}[language=Mathematica]
Manipulate[
 ContourPlot3D[(1/a) E^(2*a*z) + 
    E^(-2 z) + (x - Sqrt[a]*y)^2 == (1 + a)/(a*b^2), {x, -10, 
   10}, {y,-10,10}, {z,-10,10}, PlotRange -> All], {b,0.01, 
  1}, {a,0.05,1}]
\end{lstlisting}
\newpage
\subsection{Figure 3}
The GroupStructure class file in \cite{S2} contains a routine which specifies both the structure equations and the structure vector field ($\Sigma_\alpha$ for the $G_\alpha$ groups). Richard Schwartz made his program to explore Sol, but a small modification to the structure routine allows us to use his program for all of the $G_\alpha$ groups as well as other Lie groups (equipped with a choice of left-invariant metric) such as Nil. If the reader would like to investigate the geodesic geometry of the $G_\alpha$ groups using Schwartz's program, they need only download Schwartz's program from \cite{S2}, insert the following subroutine into the GroupStructure class file, and quickly make the necessary modifications to the user interface.
\begin{lstlisting}[language=Java]
/**For G Alpha groups, where "A" denotes the choice of alpha*/

/**This specifies the group law*/
    public static Vector leftTranslationGroup(Vector V,Vector W) {
	double x=V.x[0];
	double y=V.x[1];
	double z=V.x[2];
	double a=W.x[0];
	double b=W.x[1];
	double c=W.x[2];
	double aa=a*Math.exp(+z)+x;
	double bb=b*Math.exp(-A*z)+y;
	double cc=c+z;
	Vector Z=new Vector(aa,bb,cc);
	return Z;			 
}
/**This specifies the structure equations*/
    public static Vector structureFieldGroup(Vector V) {
	double x=V.x[0];
	double y=V.x[1];
	double z=V.x[2];
	double a=+x*z;
	double b=-A*y*z;
	double c=+A*y*y-x*x;
	Vector W=new Vector(a,b,c);
	return W;
    }
\end{lstlisting}
\subsection{Figure 4, 5, and 8}
The following renders the symmetric flowline associated to any choice of $x_0 \in (\frac{1}{\sqrt{3}}, 1)$ for the $G_{1/2}$ group. Moreover, this provides numerical evidence that the unwieldy expression for the period function in $G_{1/2}$ we presented earlier is indeed correct. Figure 5 just adds the straight line from the origin to the endpoint of the flowline to show that the "Bounding Triangle Theorem" does not work in general while Figure 8 just plots the endpoints of the flowlines.
\begin{lstlisting}[language=Mathematica]

(*This is the period function for the G_{1/2} group, 
as in Corollary 3.4 *)

P[B_]:= (*as in the statement of the corollary *)

(*This is the period function, 
after the change of variables from \beta to x0*)

L1[x0_] := P[((3*Sqrt[3]/2)*(x0-x0^3))^(1/3)]
              
 (*This numerically solves the fundamental system of ODE's *)
 
 Manipulate[
 	s1 = NDSolve[{x'[t] == -x[t]*z[t], y'[t] == (1/2)*y[t]*z[t], 
    z'[t] == -(1/2)*y[t]^2 + x[t]^2, a'[t] == 2*x[t] + a[t]*z[t], 
    b'[t] == 2 y[t] - (1/2)*b[t]*z[t], z[0] == 0, a[0] == 0, 
    b[0] == 0, x[0] == x0, y[0] == Sqrt[1 - x0^2]}, {x, y, z, a, 
    b}, {t, 0, Re[L1[x0]]/2}], {x0, 1/Sqrt[3], 1}]
    
 (*This plots the flowline*)    
 
    ParametricPlot[{a[t], b[t]} /. s1, {t, 0,Re[L1[x0]/2]}, 
 PlotRange -> All]
\end{lstlisting}
\newpage
\subsection{Figures 6 and 7}
We can use Stephen Miller's program to compute the period for arbitrary alpha, or we can "find" it by inspection, preferably looking at the function $z$.
\begin{lstlisting}[language=Mathematica]
(*This numerically solves our ODE's, d is alpha here*)

Manipulate[
 s = NDSolve[{x'[t] == -x[t]*z[t], y'[t] == d*y[t]*z[t], 
    z'[t] == -d*y[t]^2 + x[t]^2, a'[t] == 2*x[t] + a[t]*z[t], 
    b'[t] == 2 y[t] - d*b[t]*z[t], z[0] == 0, a[0] == 0, b[0] == 0, 
    x[0] == c1, y[0] == Sqrt[1 - c1^2]}, {x, y, z, a, b}, {t, 0, 
    20}], {c1, Sqrt[d/(d + 1)], 1}, {d, 0, 1}]
    
(*We can use this to find \rho by inspection*)
    Manipulate[z[t] /. s, {t, 0,20}]

(*This plots the derivative of b*)
Plot[b'[t] /. s, {t, 0, 15.99}, AspectRatio -> Automatic, 
 AxesOrigin -> {0, 0}]
\end{lstlisting}
\subsection{Figure 9}
This code provides the numerical evidence for Lemma 5.11.
\begin{lstlisting}[language=Mathematica]
L1[x0_]:= (*As before*)

(*First we evaluate the derivative*)
D[L1[x0],x0]

(*Then, we plot G(x0)*)
Plot[% - 3 (1/(2 Sqrt[x0]) + 2*x0*Sqrt[x0]/(1 - x0^2)),
{x0,1/Sqrt[3.],1},AxesOrigin -> {1/Sqrt[3.], 0}]
\end{lstlisting}
\newpage
\subsection{Computing the General Period Function}
Stephen Miller helped us with this code. It allows us to compute the period function for any choice of $\alpha$ and $\beta$. Using it, it appears that the monotonicity results we would like are indeed true for arbitrary $\alpha$, a promising sign for our Main Conjecture.
\begin{lstlisting}[language=Mathematica]
(*Defining the integrand*)
integrand[t_, A_, B_] = 2/
 Sqrt[1 - B^2/(A + 
     1) (A Exp[2 t] + Exp[-2 A t])]
     
    (*Numerically finding the endpoints of integration *)
     endpoints[A_, B_] := 
 Sort[Log[Flatten[
     y /. NSolve[
       1 - B^2/(A + 1) (A y + y^-A) == 0, 
       y, 20]]]/2]
       
     (*Numerical integration*)
       p[A_, B_] := 
 NIntegrate[integrand[t, A, B], 
  Join[{t}, endpoints[A, B]]]
  
  (*Generates the table presented earlier*)
Table[{a, p[a, .999], Pi*Sqrt[2/a]}, {a, 0.1, 1, 0.1}] 
\end{lstlisting}

\subsection{Verification of Inequality $(20)$}
Mathematica and other computer algebra systems can prove inequalities involving rational functions fairly quickly. A human-readable proof could in practice be generated as well, but an output of "true" suffices for our purposes.

\begin{lstlisting}[language=Mathematica]
Reduce[ForAll[x0, 
  1/Sqrt[3] < x0 < 
   1, ...The expression in (20)... < 1]]
(* The output to this code is true*)
\end{lstlisting}


\begin{thebibliography}{10}

\bibitem{AQ}
Anderson, G.D., Vamanamurthy, M.K., Vuorinen, M. (1992).
\textit{Functional Inequalities for Hypergeometric Functions and Complete Elliptic Integrals}
SIAM J. Math. Anal. Vol 23, pp. 512-524.

\bibitem{VA}
  Arnold, V. (1966).
  \textit{Sur la g\'eom\'etrie diff\'erentielle des groupes de Lie de dimension infinie et ses applications \`a l'hydrodynamique des fluides parfaits.} Ann. Inst. Fourier (Grenoble)
  
  \bibitem{LB}
  Bianchi, L. (1898).
  \textit{Sugli spazi a tre dimensioni che ammettono un gruppo continuo di movimenti}. 
  Memorie di Matematica e di Fisica della Societa Italiana delle Scienze, Serie Terza, Tomo XI, pp. 267-352 
  
    \bibitem{MS}
  Coiculescu, M.P., Schwartz, R.E. (2020).
  \textit{The Spheres of Sol}.
  arXiv:1911.04003v7
  
  \bibitem{CV}
  Cveticanin, L. (2004).
  \textit{Vibrations of the nonlinear oscillator with quadratic nonlinearity}.
  Physica A, 341, 123-135.
  
  \bibitem{DLMF} NIST Digital Library of Mathematical Functions. http://dlmf.nist.gov/, Release 1.0.28 of 2020-09-15. F. W. J. Olver, A. B. Olde Daalhuis, D. W. Lozier, B. I. Schneider, R. F. Boisvert, C. W. Clark, B. R. Miller, B. V. Saunders, H. S. Cohl, and M. A. McClain, eds.
  
  \bibitem{EM}
  Ellis, G.F.R., MacCallum M.A.H. (1969).
  \textit{A Class of Homogeneous Cosmological Models}. 
  Commun. math. Phys. 12, 108-141
  
  \bibitem{GPP}
  Gill, P.M., Pearce, C.E.M., Pe\v{c}ari\'{c}, J. (1997).
  \textit{Hadamard's Inequality for $r$-Convex Functions}
  Journal of Mathematical Analysis and Applications. 215, 461-470
  
  \bibitem{GR}
  Gradshteyn, I.S., Ryzhik, I.M. (2007).
  \textit{Table of Integrals, Series, and Products, Seventh Edition}
  Academic Press, Elsevier.
  
    \bibitem{G}
  M. Grayson, (1983).
  {\it Geometry and Growth in Three Dimensions\/}, Ph.D. Thesis,
  Princeton University.
  
   \bibitem{K}
   Kobayashi, S., Nomizu, K. (1963).
  \textit{Foundations of Differential Geometry Volume 1}.
  Wiley Classics Library
  
  \bibitem{N}
   Kobayashi, S., Nomizu, K. (1966).
  \textit{Foundations of Differential Geometry Volume 2}.
  Wiley Classics Library
  
  \bibitem{S}
  Schwartz, R.E. (2020).
  \textit{On Area Growth in Sol}.
  arXiv:2004.10622v2 
  \bibitem{S2}
  Schwartz, R.E. {\it Java Program for Sol\/}, download (in 2020) from \newline
http://www.math.brown.edu/$\sim$res/Java/SOL.tar
  
  \bibitem{ST}
  Trettel, S. (2019).
  \textit{Families of Geometries, Real Algebras, and Transitions}, Ph.D. Thesis,
  University of California Santa Barbara.
  
   \bibitem{T}
 Troyanov, M. (1998).
 \textit{L'horizon de Sol}. 
 Exposition. Math. 16, no. 5, 441-479.
  
  \bibitem{WW}
  Wilczynska, M.R., Webb, J.K., et al. (2020).
  \textit{Four direct measurements of the fine-structure
constant 13 billion years ago}.
 Sci. Adv. 6.
 
\end{thebibliography}
\end{document}